\documentclass[11pt]{amsart}

\usepackage{amsmath, amsthm, amssymb}
\usepackage{mathabx}
\usepackage[hidelinks]{hyperref}
\usepackage{mathrsfs}
\usepackage{xcolor, changepage} 
\usepackage{graphicx}
\usepackage[title]{appendix}
\usepackage{enumerate}
\usepackage{accents}
\usepackage{tikz-cd}
\usepackage{caption}
\usepackage{subcaption} 

\usepackage{geometry}
\newtheorem{theorem}{Theorem}[section]
\newtheorem{lemma}[theorem]{Lemma}

\newtheorem{prop}[theorem]{Proposition}

\newtheorem{claim}[theorem]{Claim}
\newtheorem{remark}[theorem]{Remark}

\newtheorem{example}[theorem]{Example}

\newtheorem{definition}[theorem]{Definition}
\numberwithin{equation}{section}

\def\bs {\boldsymbol}
\def \ed{{\rm d}}

\def\e{{\rm e}}
\def\alg{{\rm alg}}
\def\diag{{\rm diag}}
\def\rank{{\rm rank}}
\def\span{{\rm span}}
\def \w{\wedge}
\def\R{\mathbb{R}}

\def \Ecal{\mathcal{E}}
\def \Pcal{\mathcal{P}}
\def \Qcal{\mathcal{Q}}

\def\Gcal{\mathcal{G}}
\def\Wcal{\mathcal{W}}

\def\Lcal{\mathcal{L}}
\def\Dcal{\mathcal{D}}

\def\GL{{\rm GL}}
\def\SL{{\rm SL}}

\def\foli{\mathscr{F}}

\def\Ab{\bs{A}}
\def\Bb{\bs{B}}

\def\lbb{[\![}
\def\rbb{]\!]}
\def\<{\langle}
\def\>{\rangle}

\def\Maple{\textsc{Maple} }
\def\Cartan{\texttt{Cartan} }

\title[Hyperbolic Monge--Amp\`ere systems]
{Hyperbolic Monge--Amp\`ere systems with $S_1=0$}

\author{Yuhao Hu}
\address{School of Mathematical Sciences, 
Shanghai Jiao Tong University, Shanghai, 200240, P. R. China
}
\email{yuhao.hu@sjtu.edu.cn}



\begin{document}

\maketitle

\begin{abstract}
	 For hyperbolic Monge--Amp\`ere systems, a well-known solution of the equivalence problem
	 yields two invariant tensors, ${S}_1$ and ${S}_2$, defined on the underlying $5$-manifold, where ${S}_2=0$
	 characterizes systems that are Euler--Lagrange. In this article, we consider the
	  `opposite' case,
	 ${S}_1 = 0$, and show that 
	 the local generality of such systems is `$2$ arbitrary functions of $3$ variables'.
	 In addition, we classify all $S_1=0$ systems with cohomogeneity at most one, which turn out to be linear up to contact transformations.
\end{abstract}

\setcounter{tocdepth}{1}
\begin{quote}
{\tableofcontents}
\end{quote}


\section{Introduction}

The classical Monge--Amp\`ere equations in the plane are second-order PDEs of the form
\begin{equation}\label{MAE}
	A(z_{xx}z_{yy} - z_{xy}^2) + B z_{xx} + 2C z_{xy} + D z_{yy} + E = 0,
\end{equation}
where the coefficients are given functions of $x,y,z,z_x,z_y$. Geometrically, any such equation is encoded
by a $5$-manifold $M$, with local coordinates $(x,y,z,p,q)$, and a differential ideal $\Ecal\subset \Omega^*(M)$ generated by the contact form $\theta = \ed z - p\ed x - q\ed y$ and the $2$-form 
\[
	\Psi = A \ed p \w\ed q + B \ed p \w \ed y + C(\ed q\w\ed y-\ed p\w\ed x) - D\ed q\w\ed x + E\ed x\w\ed y.
\]
The resulting \emph{Monge--Amp\`ere system} $(M, \Ecal)$ is a natural setting for studying  
Monge--Amp\`ere equations modulo contact transformations.

Regarding Monge--Amp\`ere systems, a fundamental problem is the `equivalence problem'
(\cite{Gardner89}, \cite{Olver95_Equiv}):
determine whether two given systems $(M_k, \Ecal_k)$ $(k = 1,2)$ are \emph{equivalent}
in the sense that there exists a diffeomorphism $f: M_1\to M_2$ satisfying
$f^* \Ecal_2 = \Ecal_1$. When this condition holds, the underlying Monge--Amp\`ere 
PDEs are contact equivalent. Ideally, solution of the equivalence problem will yield a
classification of the equivalence classes.

In the following, we will focus on \emph{hyperbolic} Monge--Amp\`ere systems (see Definition~\ref{MADef3types}). For such systems, it is often useful to adopt a classification scheme based on the two `characteristic distributions' 
$\Dcal_+, \Dcal_-$ associated to $(M,\Ecal)$ (see \cite{Morimoto97}).
Both $\Dcal_+$ and $\Dcal_-$ are rank-$2$ distributions on $M$; 
the ranks of the successive derived distributions
$\Dcal^{(1)}_\pm, \Dcal^{(2)}_\pm,\Dcal^{(3)}_\pm$ are invariants of the system. In the `generic' case when these ranks are $(3,5,5)$ for both $\Dcal_\pm$,
Marvan--Vinogradov--Yumaguzhin \cite{MVY07} obtained a list of differential invariants that resolves the equivalence problem.

Moreover, classification has been obtained under less generic assumptions, for example, 
when $(M,\Ecal)$ admits transitive infinitesimal symmetry \cite{Tchij99}, and when $(M,\Ecal)$
arises from a system defined on a symplectic $4$-manifold
\cite{Kruglikov99}; there are also results that
 determine when a Monge--Amp\`ere equation is contact equivalent to a 
linear one \cite{Kushner08}. Extensive treatment of the classification 
of Monge--Amp\`ere equations include \cite{LRC93} and \cite{KLR07}; for a quick summary of 
known results, see \cite{Kushner09}.

The current work, however, follows a somewhat different approach to the 
equivalence problem for hyperbolic Monge--Amp\`ere systems. In \cite[Section 2.1]{BGG}, Bryant--Griffiths--Grossman, by using $G$-structure theory and coframe adaptations,
found that to any hyperbolic Monge--Amp\`ere system $(M,\Ecal)$ is associated a pair of invariant tensors $S_1$ and $S_2$ defined on $M$ (see Theorem~\ref{BGG_Thm} and equation \eqref{S1S2Tensor} below). Two outstanding features of these tensors are: 
\begin{enumerate}[\qquad (a)]
\item $(M,\Ecal)$ corresponds to the wave equation $z_{xy} = 0$ if and only if $S_1=S_2=0$;
\item $(M,\Ecal)$ is Euler--Lagrange (i.e., whose integral manifolds are critical points of a Lagrangian functional) if and only if $S_2 =0$.
\end{enumerate}
The objective of the current work is to understand the case of $S_1 = 0$, an `opposite' to the 
Euler--Lagrange case. Due to the use of Cartan--K\"ahler theory in our analysis, we will always assume that the systems considered are real analytic.

Our first result is the following (see Proposition~\ref{S2dgnrProp} below):
\begin{quote}
	\textbf{Theorem A.} \emph{If $S_1 = 0$, then $S_2$ must be degenerate.}
\end{quote}
Since the case of $S_1=S_2=0$ is well-understood, it remains to consider
the case when the coefficient matrix of $S_2$ (see \eqref{S1S2_Mat}) has rank $1$. To proceed, we put $S_2$ in a normal form,
thereby reducing to a $G_1$-structure on $M$, where $G_1\subset \GL(5,\R)$ is a $4$-dimensional
subgroup. It turns out that there are only two relative invariants $Q_1, Q_2$ at this stage (see Theorem~\ref{S1VanThm}), and three subcases arise:
\begin{enumerate}[\qquad I.]
	\item $Q_1 = Q_2 = 0$;
	\item $Q_1\ne 0, Q_2 = 0$ (or equivalently $Q_1 =0, Q_2\ne 0$);
	\item $Q_1, Q_2\ne 0$.
\end{enumerate}

Our second result gives the local generality for each of these three cases
(combining Theorem~\ref{CaseICoordForm_Thm}, Theorem~\ref{CaseII_GenThm} and Theorem~\ref{CaseIII_GenThm} below):
\begin{quote}
	\textbf{Theorem B.} \emph{The precise generality of the germs of hyperbolic Monge--Amp\`ere systems in each of the cases
	I, II and III are the following.}
	\[\renewcommand\arraystretch{1.1}
	\begin{array}{c|l}
		\mbox{Case}  & \mbox{\quad generality (up to diffeomorphism)}\\\hline
		\mbox{I} 	&	\mbox{$1$ arbitrary function of $2$ variables}\\
		\mbox{II} &		\mbox{$1$ arbitrary function of $3$ variables}\\
		\mbox{III}&		\mbox{$2$ arbitrary functions of $3$ variables}
	\end{array}
	\]
	\emph{As a result, the
	$S_1=0$ hyperbolic Monge--Amp\`ere systems locally 
depend on $2$ arbitrary functions of $3$ variables, up to diffeomorphism.}
\end{quote}
Thus, the generality in the $S_1 = 0$ case is considerably less
than that in the Euler--Lagrange case ($1$ arbitrary function of $5$ variables).

For Case I, whose structure is the simplest among the three, we show that, up to contact transformations,
the underlying Monge--Amp\`ere PDEs are precisely linear equations of the 
following form (see Theorem~\ref{CaseICoordForm_Thm}):
	\begin{equation}
		\begin{aligned}
		&z_{xy} + (\Phi_y - f_y) z_x - (\Phi_x+f_x) z_y \\
		&\qquad\qquad-
		 \big[(\Phi_x+f_x)(\Phi_y - f_y) + f_{xy}\big] z = 0,
		\end{aligned}
	\end{equation}
	where $f = f(x,y)$ is an arbitrary function of two variables, and $\Phi(x,y)$ is any function that
	satisfies
	$
		\Phi_{xy} = \e^{2f}.
	$
For these equations, there is an associated surface geometry that encodes all the structure invariants (see Section~\ref{Case_I_asociGeom}).

Our last result concerns the symmetry of the hyperbolic Monge--Amp\`ere systems
being considered (see Theorem~\ref{CohomogThm} below):
\begin{quote}
	\textbf{Theorem C.} \emph{Given a hyperbolic Monge--Amp\`ere system $(M,\Ecal)$
	satisfying $S_1 = 0$ and $S_2\ne 0$, if the Lie algebra of infinitesimal symmetry $\mathfrak{l}\subset \mathfrak{X}(M)$ satisfies $\dim(\mathfrak{l}_p)\ge 4$
	 for each $p\in M$, then $(M,\Ecal)$ must belong to Case I.}
\end{quote}
Regarding systems in Case I, we determine the co-dimension of $\mathfrak{l}_p\subset T_pM$ in terms of
the function $f(x,y)$ (see Proposition~\ref{CaseIsym_via_f}).
The lack of symmetric examples in the relatively more generic cases II and III serves as 
another contrast with the Euler--Lagrange class. 

Most of our results are obtained by using Cartan's method of equivalence, which yields
structure equations on which further analysis is based. The structure equations for Case I
 are \eqref{Q1Q2VanStrEqn} and \eqref{Q1Q2VanStrEqnConn}, while those for Cases II and III,
 which are lengthier, are included in Appendix~\ref{Apdx:StrEqn}. 

Computations in this work are assisted by the \Cartan package in \Maple.


\section{The invariant tensors $S_1$ and $S_2$}

In this section, we briefly recall Bryant--Griffiths--Grossman's solution \cite{BGG} of the equivalence problem for hyperbolic 
Monge--Amp\`ere systems, particularly the invariant tensors $S_1$ and $S_2$ 
that arise from their solution.

\begin{definition}[\cite{BGG}]\label{hMADef}
	Given a contact manifold
	 $(M^{2n+1}, \<\theta\>)$ and an $n$-form $\Psi\in \Omega^n(M)$,
	 let $\Ecal\subset\Omega^*(M)$ be the ideal generated by 
	 $\theta, \ed\theta$ and $\Psi$. The pair $(M,\Ecal)$ is called 
	 a \emph{Monge--Amp\`ere system}.
\end{definition}

In the following, we will be concerned with the classical case of $n = 2$.

\begin{definition}\label{MADef3types}
	A Monge--Amp\`ere system $(M^5; \Ecal)$, where $\Ecal = \<\theta,\ed\theta,\Psi\>_{\alg}$,
	is said to be \emph{hyperbolic} (resp., elliptic, parabolic) if $\Psi\notin \<\theta, \ed\theta\>$ and
	the quadratic equation for $\mu$ obtained from
	\[
		(\mu \ed\theta + \Psi)\w (\mu \ed\theta + \Psi)\equiv 0\mod \theta
	\]
	has \emph{positive}  (resp., negative, zero) discriminant.
\end{definition}

For hyperbolic Monge--Amp\`ere systems,  \cite{BGG} obtained the following.

\begin{theorem}[\cite{BGG}]\label{BGG_Thm}
Given a hyperbolic Monge--Amp\`ere system $(M^5,\Ecal)$ and  
	 any $x\in M$, there exists a neighborhood $U$ of $x$ and a coframe field 
	 $(\omega^0,\omega^1,\ldots,\omega^4)$ defined on $U$ such that
	 $\Ecal = \<\omega^0,\omega^1\w\omega^2,\omega^3\w\omega^4\>_{\alg}$ 
	 and
	 \begin{equation}\label{hMA_1ststrEqn}
	 \begin{aligned}
	 	\ed\left(\begin{array}{c}
			\omega^0\\
			\omega^1\\
			\omega^2\\
			\omega^3\\
			\omega^4
		\end{array}\right)
		 =& -\left(\begin{array}{ccccc}
		 	\phi_0 &&&&\\
			&\phi_1&\phi_2&&\\
			&\phi_3&\phi_4&&\\
			&&&\phi_5&\phi_6\\
			&&&\phi_7&\phi_8
			\end{array}
		 \right)\w
		 \left(\begin{array}{c}
			\omega^0\\
			\omega^1\\
			\omega^2\\
			\omega^3\\
			\omega^4
		\end{array}\right)\\
		&
		\qquad\qquad + \left(
			\begin{array}{c}
				\omega^1\w\omega^2+\omega^3\w\omega^4\\
				(V_1+V_5)\omega^0\w\omega^3+(V_2+V_6)\omega^0\w\omega^4\\
				(V_3+V_7)\omega^0\w\omega^3+(V_4+V_8)\omega^0\w\omega^4\\
				(V_8-V_4)\omega^0\w\omega^1+(V_2-V_6)\omega^0\w\omega^2\\
				(V_3 - V_7)\omega^0\w\omega^1+(V_5-V_1)\omega^0\w\omega^2
			\end{array}
		\right)
		\end{aligned}
	 \end{equation}
	 for some functions $V_1,\ldots, V_8$ and $1$-forms $\phi_i$ $(i = 0,\ldots,8)$ satisfying $\phi_0 = \phi_1+\phi_4 = \phi_5+\phi_8$.
\end{theorem}

The local coframe fields indicated in Theorem~\ref{BGG_Thm} are precisely the local sections of a $G$-structure
	$\Gcal$ on $M$,
	 where $G\subseteq \GL(5,\R)$ is the subgroup generated by
	 \[
	 	h = \left(\begin{array}{ccc}
			a & &\\
			&\Ab & \\
			&&\Bb
		\end{array}\right), \quad \Ab, \Bb\in \GL(2,\R), ~a = \det(\Ab) = \det(\Bb)
	 \] 
	 and
	 \[
	 	J =  \left(\begin{array}{ccc}
			1 & 0&0\\
			0&0 & I_2\\
			0&I_2&0
		\end{array}\right).
	 \]
	 The group $G$ acts on $\Gcal$ by $u\cdot g =  g^{-1}u$, for each coframe $u\in \Gcal$ viewed
	 as a column of $1$-forms and $g\in G$. Moreover, the structure equations on $\Gcal$ take the same form
	 as \eqref{hMA_1ststrEqn}, where
the $\omega^i$'s are the tautological $1$-forms, $\phi_\alpha$'s the pseudo-connection forms,
and $V_k$'s the torsion functions, all defined on $\Gcal$. In \cite{BGG}, it is shown how the torsion functions $V_k$ $(k = 1,2,\ldots,8)$ 
transform along the fibres of $\Gcal$; that is, by writing
\begin{equation}\label{S1S2_Mat}
	S_1 = \left(\begin{array}{cc}
						V_1&V_2\\V_3&V_4
					\end{array}\right),\qquad
	S_2 = \left(\begin{array}{cc}
						V_5&V_6\\ V_7&V_8
					\end{array}\right),
\end{equation}
we have
\begin{equation}\label{GActIdComp}
	S_i(u\cdot h) = a \Ab^{-1} S_i(u) \Bb,\quad (i = 1,2),
\end{equation}
and 
\begin{equation}\label{GActJ}
	S_1(u\cdot J) = \left(\begin{array}{cc}
						-V_4&V_2\\V_3&-V_1
					\end{array}\right)(u),
		\qquad
	S_2(u\cdot J) = \left(\begin{array}{cc}
						V_8&-V_6\\-V_7&V_5
					\end{array}\right)(u).
\end{equation}
This yields the following
invariant tensors%
\footnote{We take the opportunity to correct a typo in equation (11) of \cite{Hu20}.}
defined on $M$, also denoted by $S_1$ and $S_2$:
\begin{equation}\label{S1S2Tensor}
\begin{aligned}
	S_1 &= V_3\,\omega^1\omega^3 - V_1\, \omega^2\omega^3 + V_4\, \omega^1\omega^4 - V_2\, \omega^2\omega^4,\\
	S_2& = V_7\,\omega^1\w\omega^3 - V_5\,\omega^2\w\omega^3 + V_8\,\omega^1\w\omega^4 - V_6\,\omega^2\w\omega^4.
\end{aligned}
\end{equation}
In the following, however, we will work with $S_1$ and $S_2$ in their matrix form \eqref{S1S2_Mat} rather than the tensor form.

The following infinitesimal version of the $G$-action will also be useful
\begin{equation}\label{GActInfi}
	\ed S_i \equiv 	\left(\begin{array}{cc}
							\phi_4 & -\phi_2\\-\phi_3&\phi_1
						\end{array}\right) 
						S_i  +
					S_i\left(\begin{array}{cc}
							 \phi_5 &\phi_6\\\phi_7&\phi_8
						\end{array}\right)
						\mod\omega^0,\ldots,\omega^4\quad (i = 1,2).
\end{equation}

In \cite{BGG}, it is shown that $S_2 = 0$ if and only if $(M,\Ecal)$ is Euler--Lagrange; that is,
the integral manifolds are precisely the stationary points of a Lagrangian functional.
In the following, we will focus on the case of $S_1 = 0$, which 
may be regarded as an `opposite' to the Euler--Lagrange case.


\section{The $S_1 = 0$ case and first reductions}

\begin{prop}\label{S2dgnrProp}
	If $(M,\Ecal)$ satisfies
	$S_1 = 0$, then $\det(S_2) = 0$.
\end{prop}

The proof of this proposition relies on a \Maple computation
that rules out the case of $\det(S_2)\ne 0$; the computation is outlined
in Appendix~\ref{Apdx:S2VanPf}.

As a result of Proposition~\ref{S2dgnrProp} and the known characterization of 
$S_1 = S_2 = 0$, the only case that remains to examine
is when $\rank(S_2) = 1$. In this case, one can apply \eqref{GActIdComp} to normalize
\begin{equation}\label{S2Norm}
	S_2 = E_{21}:=\left(\begin{array}{cc}0&0\\1&0\end{array}\right).
\end{equation}
By \eqref{GActInfi}, the result of this reduction is a principal $G_1$-bundle
$\Gcal_1\subset \Gcal$ on which
\[
	\phi_1+\phi_5, \quad \phi_2, \quad\phi_6
\]
are semi-basic forms; thus, there exist functions $P_{ij}$ 
$(i = 2,5,6; j = 0,1,\ldots, 4)$ defined on $\Gcal_1$ such that
\[
	\begin{aligned}
		\phi_2& = P_{2j}\omega^j,\qquad \phi_6 = P_{6j}\omega^j,\\
		\phi_5& = -\phi_1 + P_{50}\omega^0 +P_{51}\omega^1 
				+(P_{52}+P_{21})\omega^2 + P_{53}\omega^3 
				+ (P_{54}+P_{63})\omega^4.
	\end{aligned}
\]
The new pseudo-connection matrix is
\begin{equation}\label{psuConn}
	\varphi = \left(\begin{array}{ccccc}
		\phi_0 &&&&\\
		&\phi_1&&&\\
		&\phi_3&\phi_0 - \phi_1&&\\
		&&&-\phi_1&\\
		&&&\phi_7&\phi_0+\phi_1
	\end{array}\right),
\end{equation}
and $P_{22}, P_{63}, P_{64}$ do not appear in the torsion. By adding suitable semi-basic forms to $\phi_1,\phi_3$ and $\phi_7$:
\begin{equation}\label{phi137Absrb}
\begin{aligned}
	\phi_1&\mapsto \phi_1 - P_{21}\omega^2 - P_{51}\omega^1,\\
	\phi_3&\mapsto \phi_3+P_{51}\omega^2,\\
	\phi_7&\mapsto \phi_7 + P_{53}\omega^4,
\end{aligned}
\end{equation}
and then treating the right-hand sides of \eqref{phi137Absrb} as the new $\phi_1,\phi_3$ and $\phi_7$, the torsion functions $P_{21}, P_{51}$ and $P_{53}$ will be completely absorbed
into the pseudo-connection. 

Among the $9$ remaining torsion functions, further relations are found by computing
\[
	\begin{aligned}
		\ed^2\omega^1&\equiv \phantom{+}P_{24}\,\omega^0\w\omega^3\w\omega^4\mod\omega^1,\omega^2,\\
		\ed^2\omega^2&\equiv -P_{54}\,\omega^0\w\omega^3\w\omega^4\mod\omega^1,\omega^2,\\
		\ed^2\omega^3&\equiv -P_{62}\,\omega^0\w\omega^1\w\omega^2\mod\omega^3,\omega^4,\\
		\ed^2\omega^4&\equiv \phantom{+}P_{52}\,\omega^0\w\omega^1\w\omega^2\mod\omega^3,\omega^4;
	\end{aligned}
\]
as a result, 
\[
	P_{24} = P_{54} = P_{62} = P_{52} = 0. 
\]
Now the structure equations take the form:
\begin{equation}\label{S1VanPreStrEqn}
	\ed\omega = -\varphi\w\omega + T
\end{equation}
where $\varphi$ takes the same form as \eqref{psuConn}, and the torsion 
\[
	T = \left(
		\begin{array}{c}
			\omega^1\w\omega^2+\omega^3\w\omega^4\\
			-P_{20}\omega^0\w\omega^2+P_{23}\omega^2\w\omega^3\\
			\omega^0\w\omega^3\\
			-\omega^0\w(P_{50}\omega^3+P_{60}\omega^4) - P_{61}\omega^1\w\omega^4\\
			P_{50}\omega^0\w\omega^4 - \omega^0\w\omega^1
		\end{array}
	\right).
\]

\begin{prop}
	The torsion functions $P_{20}, P_{50}, P_{60}$ must all vanish.
\end{prop}
\begin{proof}
	Restricting the structure equations \eqref{S1VanPreStrEqn} to 
	any local section of $\Gcal_1$, there exist functions
	$P_{ij}$ $(i = 0,1,3,7; j =0,1\ldots,4)$ such that
	\[
		\phi_i = P_{ij}\omega^j \quad (i = 0,1,3,7).
	\]
	Since $P_{31}$ and $P_{73}$ belong to terms that are annihilated in 
	$\varphi\w\omega$, there are
	$23$ $P_{ij}$'s remaining in the equations; for them, define $P_{ijk}$ 
	$(k = 0,1,\ldots, 4)$ by
	\[
		\ed P_{ij} = P_{ijk}\omega^k.
	\]
	
	Applying $\ed^2 = 0$ to $\omega^i$ $(i = 0,1,\ldots, 4)$ yields a system of $41$ distinct polynomial 
	equations, the solution of which expresses $39$ of the $P_{ijk}$'s in terms of the
	remaining variables in the system. Thus, at this stage, $23\times 5 - 39 = 76$ $P_{ijk}$'s are `free'.
	
	Now we consider the Pfaffian system $I$ generated by the twenty-three $1$-forms $\ed P_{ij}  - P_{ijk}\omega^k$, denoted by $\theta^\alpha$. 
	Differentiating $\theta^\alpha$ yields equations of the form
		\begin{equation}\label{dPijPfaff}
			\ed \theta^\alpha \equiv A^\alpha_{\sigma j} \pi^\sigma \w\omega^j + 
								\frac{1}{2}C^\alpha_{jk}\omega^j\w\omega^k \mod I,
		\end{equation}
	where all $A^\alpha_{\sigma j}$ are constants, and the $\pi^\sigma$'s
	are the exterior derivatives of the $76$ `free' $P_{ijk}$.
	Then a computation shows that the torsion in \eqref{dPijPfaff} is absorbable (via adding semi-basic forms
	to $\pi^\sigma$) if and only if 
	$P_{20}, P_{50}$ and $P_{60}$ are zero.
\end{proof}

Rewriting $Q_1 = P_{23}$ and $Q_2 = P_{61}$, we have obtained the following.

\begin{theorem}\label{S1VanThm}
	A hyperbolic Monge--Amp\`ere system $(M,\Ecal)$ satisfies $S_1 = 0$ and $S_2\ne 0$
	if and only if each $x\in M$ has a neighborhood $U$ on which
	there exists a coframe field $(\omega^0,\omega^1,\ldots,\omega^4)$, 
	some $1$-forms $\phi_0, \phi_1, \phi_3, \phi_7$, and functions $Q_1, Q_2$ that satisfy
	\[
		\Ecal = \<\omega^0,\omega^1\w\omega^2,\omega^3\w\omega^4\>_\alg
	\]
	and the structure equations
	\begin{equation}\label{S1VanstrEqn}
		\begin{aligned}
			\ed\left(\begin{array}{c}
			\omega^0\\
			\omega^1\\
			\omega^2\\
			\omega^3\\
			\omega^4
		\end{array}\right) 
		&= -\left(
			\begin{array}{ccccc}
				\phi_0 &&&&\\
				&\phi_1&&&\\
				&\phi_3&\phi_0 - \phi_1&&\\
				&&&-\phi_1 &\\
				&&&\phi_7&\phi_0+\phi_1
			\end{array}
		\right)\w\left(\begin{array}{c}
			\omega^0\\
			\omega^1\\
			\omega^2\\
			\omega^3\\
			\omega^4
		\end{array}\right) \\
		& \qquad\qquad
			+\left(
			\begin{array}{c}
				\omega^1\w\omega^2+\omega^3\w\omega^4\\
				Q_1 \,\omega^2\w\omega^3\\
				\phantom{+}\omega^0\w\omega^3\\
				Q_2 \, \omega^4\w\omega^1\\
				-\omega^0\w\omega^1
			\end{array}
		\right).
		\end{aligned}
	\end{equation}
\end{theorem}

Coframe fields that satisfy the structure equations \eqref{S1VanstrEqn} are precisely local sections of the principal $G_1$-bundle $\Gcal_1$,
where $G_1\subset \GL(5,\R)$ is the stabilizer of $E_{21}$ (see \eqref{S2Norm})
under the action of $G$. More explicitly, by \eqref{GActIdComp} and \eqref{GActJ}, $G_1$ is generated by 
\[
	h  = \left(\begin{array}{ccccc}
		a&&&&\\
		&A_1&&&\\
		&A_3&a/A_1&&\\
		&&&1/A_1&\\
		&&&B_3&aA_1
	\end{array}\right) \quad\mbox{where } a, A_1\ne 0 \mbox{ and }  A_3, B_3\in \R
\]
and
\[
	\tilde J = \diag(-1,1,-1,-1,1) J.
\]
By taking the exterior derivative of the equations \eqref{S1VanstrEqn}
and computing the action of $\tilde J$ directly, we find
\begin{equation}\label{QtransIdComp}
\left.
\begin{aligned}
	\ed Q_1 &\equiv (\phi_0 - 3\phi_1)Q_1 \\
	\ed Q_2 & \equiv (\phi_0+3\phi_1)Q_2 
\end{aligned}
~\right\}  \mod \omega^0,\ldots,\omega^4,
\end{equation}
and, for $u\in \Gcal_1$,
\begin{equation}\label{QtransJtilde}
	(Q_1, Q_2)(u\cdot \tilde J) = (Q_2, - Q_1)(u).
\end{equation}
In particular, $\tilde J$ acts on $(Q_1,Q_2)\in \R^2$ as a clockwise rotation by $\pi/2$.
It is easy to see that the action of $G_1$ on $\R^2$ has 3 orbits, which gives rise to the three cases that we will examine below:
\begin{equation}\label{threeCases}
\mbox{I:}~\;  Q_1 = Q_2 = 0; \qquad\quad
\mbox{II:}~\;  Q_1\ne 0,~ Q_2 = 0;\qquad\quad
\mbox{III:}~\;  Q_1, Q_2\ne 0. 
\end{equation}

\section{Case I: $Q_1 = Q_2= 0$}

In this case, the structure equations \eqref{S1VanstrEqn} read
\begin{equation}\label{Q1Q2VanStrEqn}
	\begin{aligned}
		\ed\omega^0& = -\phi_0\w\omega^0 + \omega^1\w\omega^2+\omega^3\w\omega^4,\\
		\ed\omega^1& = -\phi_1\w\omega^1,\\
		\ed\omega^2& = -\phi_3\w\omega^1 - (\phi_0-\phi_1)\w\omega^2
						+\omega^0\w\omega^3,\\
		\ed\omega^3& = \phantom{+}\phi_1 \w\omega^3,\\
		\ed\omega^4& = -\phi_7\w\omega^3 - (\phi_0+\phi_1)\w\omega^4
					- \omega^0\w\omega^1.
	\end{aligned}
\end{equation}
By taking the exterior derivative of \eqref{Q1Q2VanStrEqn}, it is easy to show that the pseudo-connection forms $\phi_0$, $\phi_1$ satisfy the equations
\begin{equation}\label{Q1Q2VanStrEqnConn}
	\begin{aligned}
	\ed\phi_0 & = 2\,\omega^1\w\omega^3,\\
	\ed\phi_1& = A\, \omega^1\w\omega^3,
	\end{aligned}
\end{equation}
for some function $A$.

\subsection{An associated geometry}\label{Case_I_asociGeom}

From \eqref{Q1Q2VanStrEqn} and 
\eqref{Q1Q2VanStrEqnConn}, one immediately observes that $\{\omega^1,\omega^3,\phi_1\}$
generates a Pfaffian system on $\Gcal_1$ whose \emph{retracting space} (see \cite{BCG}) is itself. This implies the existence of a local submersion from $\Gcal_1$ to a 
$3$-manifold $\Wcal$ on which $\omega^1,\omega^3,\phi_1$ descend to form a local coframe field.
In fact, $\Wcal$ can be regarded as the structure bundle, with $(\omega^1,\omega^3)$ being the tautological forms, for the following data 
defined on an open subset $\Sigma$ of $\R^2$: an area form $\Omega = \omega^1\w\omega^3$, and an unordered pair of `transverse foliation by curves'
 $\foli_1$ and $\foli_2$ defining a \emph{net} on $\Sigma$, where $\foli_1$ and $\foli_2$
 are integral curves of $\ker(\omega^3)$ and $\ker(\omega^1)$, respectively. 
These are precisely the data%
\footnote{Robert Bryant informed the author of this fact.} needed to specify a Lorentzian metric $\mbox{\fontshape{ui}\selectfont g}= \omega^1\circ\omega^3$, defined up to sign, on an oriented surface $\Sigma$.
In view of these relations, we will call $(\Sigma, \Omega; \foli_1,\foli_2)$ the 
\emph{associated geometry} of a hyperbolic Monge--Amp\`ere system $(M,\Ecal)$ in Case I. Note that there is a canonical local submersion $\pi: (M,\Ecal)\to (\Sigma, \Omega;\foli_1, \foli_2)$.
Also note that the invariants of $(M,\Ecal)$ are completely encoded by
its associated geometry.

\subsection{Integration of the structure equations}\label{CaseI_intg_Section}
Now we proceed to 
integrate the structure equations \eqref{Q1Q2VanStrEqn} and \eqref{Q1Q2VanStrEqnConn} and 
determine the underlying hyperbolic Monge--Amp\`ere PDEs in coordinate form, up to contact transformations. 

We begin by noticing that both $\omega^1$ and $\omega^3$ are integrable; thus, locally 
there exist functions $x,y,f,g$ such that
\[
	\omega^1 = \e^f \ed x, \qquad \omega^3 = \e^g \ed y.
\]
Furthermore, due to the freedom of scaling $(\omega^1,\omega^3)\mapsto (\lambda\omega^1,\lambda^{-1}\omega^3)$ while adjusting $\phi_1$ by a multiple of $\ed\lambda$ without affecting the 
structure equations, we can arrange that 
\[
	g = f.
\]

Now, the equation of $\ed\phi_0$ implies that 
the $2$-form $\omega^1\w\omega^3 = \e^{2f} \ed x\w\ed y$ is closed; hence, $f$ is a function of $x,y$. Then, integrating the equation of $\ed\phi_0$ shows that there exists a function $S$ 
such that
\[
	\phi_0 = - \Phi_x\ed x + \Phi_y\ed y + \ed S,
\]
where $\Phi(x,y)$ is a function that satisfies
\begin{equation}\label{Phi}
	\Phi_{xy} = \e^{2f}.
\end{equation}

Now the equations of $\ed\omega^1$ and $\ed\omega^3$ determine 
\[
	\phi_1 = f_x\ed x - f_y \ed y,
\]	
and then by the equation of $\ed \phi_1$, the invariant $A$ has the expression
\[
	A = -\frac{2f_{xy}}{\e^{2f}}.
\]

By \eqref{Q1Q2VanStrEqn}, the Pfaffian systems generated by 
$\{\omega^1,\omega^2,\omega^3\}$
and $\{\omega^1,\omega^3,\omega^4\}$, respectively, are both Frobenius. Thus, 
there exist functions $p,q$ such that
\[
	\lbb\omega^1,\omega^2,\omega^3\rbb = \lbb\ed x, \ed y, \ed p\rbb,
	\qquad
	\lbb\omega^1,\omega^3,\omega^4\rbb = \lbb\ed x, \ed y, \ed q\rbb,
\]
where $\lbb\ldots\rbb$ indicates the span of the forms within.
We still have the freedom of applying the following transformations
\begin{equation}\label{CaseIom024trans}
	(\omega^0,\omega^2,\omega^4)\mapsto (\lambda\omega^0, \lambda\omega^2+\mu_1\omega^1, \lambda\omega^4+\mu_2\omega^3)
\end{equation}
while adjusting $\phi_0, \phi_3, \phi_7$ accordingly, without affecting the structure equations. Using this,
we can arrange that
\begin{equation}\label{CaseIom2om4}
	\omega^2 = \alpha\ed y + \e^T \ed p,\qquad
	\omega^4 = \beta\ed x + \e^{-T} \ed q
\end{equation}
for some functions $\alpha,\beta$ and $T$.
Substituting these in the equations of $\ed\omega^2, \ed\omega^4$ in 
\eqref{Q1Q2VanStrEqn} and reducing modulo $\ed x, \ed y$, we obtain
\[
\begin{aligned}
	-\e^T \ed(S + T)\w \ed p \equiv \e^{-T} \ed(S - T)\w\ed q\equiv 0 \mod \ed x, \ed y;\\
\end{aligned}
\]
this implies that $S+T$ is a function of $x,y,p$, and $S - T$ is a function of $x,y,q$.

\begin{claim}
	By suitable choices of $p$ and $q$, we can arrange that $S = T = 0$.
\end{claim}
\begin{proof}
	Consider a change of variables $p = \Pcal(x,y,\tilde p)$ and $q = \Qcal(x,y, \tilde q)$. Substituting these in 
	\eqref{CaseIom2om4}, applying transformations of the form \eqref{CaseIom024trans}, 
	and using the equation of $\ed\omega^0$, we obtain
	\[
		\tilde \omega^2 = \tilde \alpha \ed y + \e^{\tilde T} \ed \tilde p, 
		\qquad
		\tilde\omega^4 = \tilde \beta \ed x + \e^{-\tilde T} \ed \tilde q,
	\]
	and 
	\[
		\tilde \phi_0 = -\Phi_x\ed x+\Phi_y\ed y+ \ed \tilde S,
	\]
	with
	\[
		\tilde S = S+\frac{1}{2}(\ln \Pcal_{\tilde p} + \ln \Qcal_{\tilde q}), \qquad
		\tilde T = T + \frac{1}{2}(\ln \Pcal_{\tilde p} - \ln \Qcal_{\tilde q}).
	\]
	Since $S+T$ is a function of $x,y,p$, one can arrange $\tilde S+\tilde T = 0$ by choosing
	a suitable $\Pcal(x,y,\tilde p)$; similarly, one can arrange $\tilde S - \tilde T  =0 $ by choosing
	a suitable
	$\Qcal(x,y,\tilde q)$. 
\end{proof}

	With $S = T = 0$, the equations of $\ed \omega^2$ and $\ed\omega^4$ imply
	\begin{equation}\label{CaseIom0Coord}
		\e^f \omega^0 \equiv \ed\alpha - (\Phi_y+f_y) \ed p \equiv 
		 -\ed\beta - (\Phi_x - f_x) \ed q   \mod \ed x, \ed y.
	\end{equation}
	Rearranging terms, we get
	\[
		\ed(\alpha+\beta) - (\Phi_y+f_y) \ed p + (\Phi_x-  f_x) \ed q \equiv 0 \mod \ed x, \ed y.
	\]
	Since both $f$ and $\Phi$ are functions of $x,y$, the congruence above implies that there exist
	functions $\kappa$ and $H(x,y)$
	such that
	\begin{equation}\label{Case_I_alphabeta}
		\begin{aligned}
			\alpha & =\phantom{+} \kappa + \frac{1}{2}\left[(\Phi_y+f_y)p - (\Phi_x - f_x)q + H\right],\\
			\beta & = -\kappa + \frac{1}{2}\left[(\Phi_y+f_y)p - (\Phi_x - f_x)q + H\right].
		\end{aligned}
	\end{equation}
	Since $\Phi_{xy} = \e^{2f}\ne 0$, $(\Phi+f)_y$ and $(\Phi - f)_x$ cannot both vanish; thus, 
	by adding a suitable function of $x,y$ to $p$ or $q$, we can arrange that $H(x,y) = 0$. 
	Then from \eqref{CaseIom0Coord} we determine $\omega^0$ modulo $\ed x$ and $\ed y$; that is,
	there exist functions $P,Q$ such that
	\begin{equation}\label{CaseIom0Coord2}
		\e^f \omega^0 = \ed Z - P \ed x - Q \ed y,
	\end{equation}
	where
	\begin{equation}\label{Case_I_Z}
	 	Z = \kappa - \frac{1}{2}(\Phi_y+f_y) p -\frac{1}{2} (\Phi_x - f_x)q.
	\end{equation}
	
	At this stage, all the $\omega^i$'s have been put in coordinate form, and the equations of 
	$\ed\omega^i$ for $i = 1,\ldots, 4$ become identities. Thus,  it remains to substitute 
	the coordinate form of the $\omega^i$'s in the equation of $\ed\omega^0$, which reduces to 
	\begin{equation}\label{Case_I_R1R2rel}
	\begin{aligned}
		\e^{-\Phi} \ed R_1 \w\ed x + \e^{\Phi} \ed R_2\w \ed y = 0,
	\end{aligned}
	\end{equation}
	where $R_1$ and $R_2$ are defined by
	\begin{equation}\label{Case_I_R1R2}
	\begin{aligned}
		P & = \e^{2f} p + (\Phi_x + f_x) Z + \e^{f - \Phi} R_1,\\
		Q& = \e^{2f} q - (\Phi_y - f_y) Z + \e^{f+\Phi} R_2.
	\end{aligned}
	\end{equation}
	
\begin{claim}
	By suitable choices of $p$ and $q$, 
	we can arrange that $R_1 = R_2 = 0$.
\end{claim}
\begin{proof}
	Putting 
	$\tilde p = p+U(x,y)$, $\tilde q  = q+ V(x,y)$
	and substituting in \eqref{CaseIom2om4} (note that we have arranged $T = 0$), one obtains
	$\tilde \alpha = \alpha - U_y$ and  $\tilde \beta = \beta - V_x$;
	then \eqref{Case_I_alphabeta} implies
	\[
		\tilde\kappa = \kappa - \frac{1}{2}(U_y - V_x).
	\]
	The condition $\tilde H = 0$ puts a constraint between $U$ and $V$:
	\begin{equation}\label{Case_I_UVconstraint}
		U_y + V_x = -(\Phi_y+f_y)U - (\Phi_x - f_x)V.
	\end{equation}
	Substituting $\tilde p, \tilde q, \tilde \kappa$ in the right-hand side of
	\eqref{Case_I_Z} determines $\tilde Z$, 
	which, via $\ed Z - P \ed x - Q\ed y = \ed \tilde Z - \tilde P \ed x  - \tilde Q\ed y$,
	determines $\tilde  P$ and $\tilde Q$; then $\tilde R_1, \tilde R_2$ can be determined 
	through \eqref{Case_I_R1R2} (for $\tilde P,\tilde Q$). Now the condition $\tilde R_1 = \tilde R_2 = 0$, together
	with \eqref{Case_I_UVconstraint}, form a system of second-order PDEs for $U,V$. 
	Since we work in the real analytic category, a standard Cartan--K\"ahler analysis shows that this system
	is involutive if and only if $R_1$ and $R_2$ satisfy
	the condition \eqref{Case_I_R1R2rel}.
	This proves the claim.
\end{proof}

Expressing $\e^f \omega^1\w\omega^2$ in the $(x,y,Z,P, Q)$-coordinates and then replacing $\ed Z$ by 
$P\ed x + Q\ed y$, we obtain
\[
\begin{aligned}
	\e^f\omega^1\w\omega^2&\equiv 
			\ed x\w\Big\{\ed P + 
			\Big[ (\Phi_y - f_y)P -(\Phi_x+f_x) Q \\
			&\qquad	\qquad\qquad\qquad  
					- \big((\Phi_x+f_x)(\Phi_y-f_y ) + f_{xy} \big)Z \Big]\ed y\Big\}
			\mod\omega^0.
\end{aligned}
\]
This completes the integration of the structure equations \eqref{Q1Q2VanStrEqn} and \eqref{Q1Q2VanStrEqnConn}. To summarize, we have obtained the following.
\begin{theorem}\label{CaseICoordForm_Thm}
	Up to contact transformations, any hyperbolic Monge--Amp\`ere system satisfying 
		$S_1 = 0$, $S_2\ne 0$ and  $Q_1 = Q_2 = 0$ locally corresponds to a linear hyperbolic Monge--Amp\`ere PDE of the form
	\begin{equation}\label{Q1Q2VanCoordForm}
		\begin{aligned}
		&z_{xy} + (\Phi_y - f_y) z_x - (\Phi_x+f_x) z_y \\
		&\qquad\qquad-
		 \big[(\Phi_x+f_x)(\Phi_y - f_y) + f_{xy}\big] z = 0,
		\end{aligned}
	\end{equation}
	where $f = f(x,y)$ is an arbitrary function of two variables, and $\Phi(x,y)$ is any function that satisfies
	\[
		\Phi_{xy} = \e^{2f}.
	\] 
	For such equations, the function
	\[
		A = -\frac{2f_{xy}}{\e^{2f}}
	\]
is a local invariant with respect to contact transformations. 
\end{theorem}

\begin{example}
	The simplest example occurs when $f = 0$, for which the invariant $A=0$.
	With the choice $\Phi(x,y) = xy$, the underlying
	hyperbolic Monge--Amp\`ere equation is
	\begin{equation}\label{Case_I_MostFlat}
		z_{xy} + xz_x - yz_y - xyz = 0.
	\end{equation}
	In a sense, this equation is the `flattest' among all hyperbolic 
	Monge--Amp\`ere equations satisfying $S_1 = 0$ and $S_2\ne 0$.
\end{example}


\section{Case II: $Q_1\ne 0$, $Q_2 = 0$}

In this case, by applying the $G_1$-action and using \eqref{QtransIdComp} and \eqref{QtransJtilde}, one can normalize
\begin{equation}\label{Q1Q2onezero}
	(Q_1, Q_2) = (1,0).
\end{equation}
By \eqref{QtransIdComp}, on the sub-bundle $\Gcal_2\subset \Gcal_1$ defined by \eqref{Q1Q2onezero}, 
\[
	\phi_0 - 3\phi_1
\]
is semi-basic; thus, there exist functions $P_{0j}$ $(j = 0,1,\ldots, 4)$ such that
\[
	\phi_0 = 3\phi_1 + P_{0j}\omega^j.
\]

By the computation
\[
	\ed^2\omega^1\equiv -(P_{00}\omega^0 + P_{04}\omega^4)\w\omega^2\w\omega^3\mod\omega^1,
\]
we obtain
\[
	P_{00} = P_{04} =0;
\]
then
\[
	\ed^2\omega^1 = - (\ed\phi_1+P_{01}\omega^2\w\omega^3-\phi_3\w\omega^3)\w\omega^1,
\]
which implies
\begin{equation}\label{CaseIIdphi1_interm}
	\ed\phi_1  \equiv (-P_{01}\omega^2+\phi_3)\w\omega^3\mod\omega^1.
\end{equation}
On the other hand, $\ed^2\omega^3 = 0$ implies
	\[
		\ed\phi_1\equiv 0\mod\omega^3.
	\]
Combining this with \eqref{CaseIIdphi1_interm}, it is easy to see that there exists a function 
$A$ 
	such that
	\begin{equation}\label{CaseIIdphi1}
		\ed\phi_1 = (-P_{01}\omega^2+\phi_3 + A\omega^1)\w\omega^3.
	\end{equation}
	
Now
\[
	\ed^2\omega^0\equiv (\ed P_{03}+P_{03}\phi_1+3\phi_3)\w\omega^0\w\omega^3
		\mod \omega^1,\omega^2,\omega^4;
\]
from this we obtain the infinitesimal transformation of $P_{03}$ along the group fiber:
\[
	\ed P_{03}\equiv -\phi_1P_{03} -3\phi_3 \mod\omega^0,\omega^1,\ldots\omega^4.
\]
Due to the `$-3\phi_3$'-term, a further reduction is thus available, and we will normalize to 
\[
	P_{03}= 0.
\]
The result is a sub-bundle $\Gcal_3\subset \Gcal_2$ on which $\phi_3$ 
becomes semi-basic, and we write
\[
	\phi_3 = P_{30}\omega^0+P_{31}\omega^1+(P_{32}+P_{01})\omega^2+P_{33}\omega^3
				+P_{34}\omega^4.
\]
By computing $\ed^2\omega^0$ modulo $\omega^1,\omega^2$, we find
\[
	P_{34} = 0.
\]

Now the structure equations read
\begin{equation}\label{CaseIIstrEqn}
	\ed\omega = -\varphi\w\omega+T,
\end{equation}
where
\[
	\varphi = \left(\begin{array}{ccccc}
		3\phi_1&&&&\\
		&\phi_1&&&\\
		&&2\phi_1&&\\
		&&&-\phi_1&\\
		&&&\phi_7&4\phi_1
	\end{array}\right)
\]
and
\[
	T = \left(
	\begin{array}{c}
		\omega^0\w(P_{01}\omega^1+P_{02}\omega^2)+\omega^1\w\omega^2+\omega^3\w\omega^4\\
		\omega^2\w\omega^3\\
		-P_{30}\omega^0\w\omega^1+\omega^0\w\omega^3+P_{32}\omega^1\w\omega^2+P_{33}\omega^1\w\omega^3\\
		0\\
		-P_{01}\omega^1\w\omega^4-P_{02}\omega^2\w\omega^3-\omega^0\w\omega^1
	\end{array}
	\right).
\]
By taking the exterior derivative of these structure equations, we find the infinitesimal transformation of 
the torsion functions along each fiber, where all congruences below are computed modulo $\omega^0,\ldots,\omega^4$:
\begin{equation}\label{CaseII5torTrans}
	\begin{array}{lll}
		\ed P_{01}\equiv \phi_1P_{01},
		\quad&\ed P_{02}\equiv \phantom{+}2\phi_1P_{02},
		\quad &\ed P_{30}\equiv 2\phi_1P_{30},\\
		\ed P_{32}\equiv \phi_1P_{32},
		&\ed P_{33}\equiv -2\phi_1P_{33}.
		&
	\end{array}
\end{equation}
Thus, all $P_{01}, P_{02}, P_{30}, P_{32}, P_{33}$ are relative invariants, which scale along the fiber. 
It is natural to ask whether they can simultaneously vanish.

\begin{lemma}\label{CaseII5torLemma}
	The torsion functions $P_{01}, P_{02}, P_{30}, P_{32}, P_{33}$ cannot all be zero.
\end{lemma}
\begin{proof}
	If all the $5$ torsion functions were zero, 
	then differentiating $\ed\omega^0$ and $\ed\omega^2$ and using 
	\eqref{CaseIIdphi1}, we would obtain
	\[
	\begin{aligned}
		\ed^2\omega^0 & = (-3P_{31} - 3A + 2)\,\omega^0\w\omega^1\w\omega^3,\\
		\ed^2\omega^2 & = (1+2P_{31}+2A)\,\omega^1\w\omega^2\w\omega^3.
	\end{aligned}
	\]
	Thus, $A$ must be equal to $-P_{31} + (2/3)$ and $-P_{31} - (1/2)$ simultaneously, which
	is impossible.
\end{proof}

By Lemma~\ref{CaseII5torLemma}, at least one of 
$P_{01}, P_{02}, P_{30}, P_{32}, P_{33}$ is nonzero, and then it can be normalized to $\pm 1$
according to \eqref{CaseII5torTrans}. Instead of analyzing the various cases that 
arise when one asks \emph{which} of these torsion functions are nonzero, we will proceed 
assuming that one of these functions (without specifying which one) is nonzero and  normalized, and the result is 
a sub-bundle $\Gcal_4\subset \Gcal_3$ on which the restriction of $\phi_1$ becomes
semi-basic:
\[
	\phi_1 = P_{1j}\omega^j.
\]
This introduces $5$ new torsion functions, among whom 
$P_{10}, P_{11}, P_{12}, P_{14}$ are invariants of the structure, and 
$P_{13}$ transforms by
\begin{equation}\label{CaseIIP13trans}
	\ed P_{13} \equiv \phi_7 P_{14}\mod\omega^0,\omega^1,\ldots,\omega^4.
\end{equation}
Thus, whether further reduction is possible depends on whether $P_{14}$ is nonzero; regarding this, 
we have the following.

\begin{prop}\label{P14VanProp}
	$P_{14}$ must vanish.
\end{prop}
\begin{proof}
The proof relies on a \Maple computation that shows that
if $P_{14}$ was nonzero, then after reducing to $P_{13} = 0$, 
the resulting structure equations (for an $e$-structure) would be non-involutive. Details of this computation
are included in Appendix~\ref{Apdx:P14VanPf}.
\end{proof}

At this stage, the remaining torsion functions are
\[
	P_{01}, P_{02}, P_{30}, P_{32}, P_{33}; P_{10}, P_{11}, P_{12}, P_{13}.
\]
In the following, we consider the cases $P_{01}\ne 0$ and $P_{01}=0$ separately.

\subsection{Case IIa. $P_{01}\ne 0$}
By \eqref{CaseII5torTrans}, we may assume that a normalization has been made such that
\[
	P_{01} = 1.
\]
Now computing 
\begin{equation}\label{CaseIIaP02}
\begin{aligned}
	\ed^2\omega^1&\equiv (4P_{10}P_{13}+P_{12} - P_{30})\,\omega^0\w\omega^1\w\omega^3\\
						&\qquad\qquad  
						- \ed P_{10}\w\omega^0\w\omega^1+ \ed P_{13}\w\omega^1\w\omega^3\mod\omega^2,\omega^4,\\
	\frac{1}{2}\ed^2\omega^2 &\equiv 
		\left( 4P_{10} P_{13} + P_{12} + \frac{1}{2}P_{02} + \frac{1}{2}P_{30}\right)
			\omega^0\w\omega^2\w\omega^3 \\
			&\qquad\qquad
			- \ed P_{10}\w\omega^0\w\omega^2
			+\ed P_{13}\w\omega^2\w\omega^3\mod\omega^1,\omega^4
\end{aligned}
\end{equation}
shows that
\[
	P_{02} = -3P_{30}.
\]
Using this and computing 
\begin{equation}\label{CaseIIaP10}
	\begin{aligned}
		\ed^2\omega^3& \equiv (2P_{10}P_{11} - P_{30}P_{12}+P_{10})\,\omega^0\w\omega^1\w\omega^3\\
				&\qquad\qquad
				+\ed P_{10}\w\omega^0\w\omega^3 + \ed P_{11}\w\omega^1\w\omega^3
				\mod\omega^2,\omega^4,\\
		-\frac{1}{4}\ed^2\omega^4&\equiv \left(2P_{10}P_{11} - P_{30}P_{12} +\frac{3}{4}P_{10}+\frac{3}{4}P_{30}^2 \right)\omega^0\w\omega^1\w\omega^4\\
				&\qquad\qquad
				+\ed P_{10}\w\omega^0\w\omega^4 + \ed P_{11}\w\omega^1\w\omega^4
				\mod\omega^2,\omega^3,
	\end{aligned}
\end{equation}
we obtain
\[
	P_{10} = 3P_{30}^2.
\]
Now there are $6$ torsion functions remaining:
\begin{equation}\label{CaseIIa_6Pij}
	P_{11}, P_{12}, P_{13}, P_{30}, P_{32}, P_{33},
\end{equation}
which are scalar invariants of the structure.

\subsubsection{The structure equations}

For the $P_{ij}$'s in \eqref{CaseIIa_6Pij}, define $P_{ijk}$ ($30$ in total) by
\[
	\ed P_{ij} = P_{ijk}\omega^k.
\]
Applying $\ed^2 = 0$ to $\omega^0,\ldots, \omega^4$ and, in particular, reducing 
the equation of $\ed^2\omega^4$ modulo $\omega^3$,
we obtain a system of $28$ polynomial equations in the $P_{ij}$ and the $P_{ijk}$'s.
The solution of this system expresses $18$ among the $P_{ijk}$'s in terms of the remaining 
variables, and thus there are $12$ $P_{ijk}$'s that are `free' at this stage. 

Then consider the Pfaffian system $I$ generated by the six
$\ed P_{ij} - P_{ijk}\omega^k$. The tableau
of $I$ is $12$-dimensional; that the torsion must
be absorbable imposes a single condition on one of the $P_{ijk}$'s (in our computation, $P_{132}$).
Updating the tableau, which is now $11$-dimensional, the torsion of $I$ can be 
completely absorbed. 

For the structure equations obtained at this stage, see \eqref{Apdx:CaseIIaStrEqn_om}
and \eqref{Apdx:CaseIIaStrEqn_As} in Appendix~\ref{Apdx:StrEqn}, where
the $P_{ij}$'s in \eqref{CaseIIa_6Pij} are renamed, in the same order, as $A_\rho$ $(\rho = 1,2,\ldots, 6)$,
while the $P_{ijk}$'s are renamed as $A_{\rho,k}$; $\phi_7$ is renamed as `$\phi$'.

Incidentally, 
the tableau of $I$ is involutive with the Cartan characters
\begin{equation}\label{CaseIIaCartanChar_crude}
	(s_1, s_2, s_3, s_4, s_5) = (6,4,1,0,0).
\end{equation}
Had the structure equations been for an $e$-structure, one could 
apply Cartan's third fundamental theorem (see \cite[Theorem~3 and Remark~8]{Bryant14}) 
and conclude right away that the germs of hyperbolic Monge--Amp\`ere systems
in case IIa depend on precisely $1$ arbitrary function of $3$ variables. However, we do not have
an $e$-structure yet; thus, more work is needed to obtain the generality. 
(See also Remark~\ref{Cartan3rdRmk} below.)

\subsubsection{The generality} 

\begin{prop}\label{CaseIIaGen}
	Locally the hyperbolic Monge--Amp\`ere systems in Case IIa 
	can be uniquely determined (up to contact transformations) by specifying $1$ arbitrary function of $3$ variables.
\end{prop}
\begin{proof}
	The proof relies on the following observation: While $A_1,A_2,\ldots, A_6$ in \eqref{Apdx:CaseIIaStrEqn_om} are scalar invariants, it is 
	still possible for their $\omega$-derivatives
	to be non-constant along the fibers of
	$\Gcal_4$. When this holds, reduction to an $e$-structure is possible; and then 
	Cartan's third fundamental theorem  \cite[Theorem~3]{Bryant14} can be applied to obtain 	
	the generality.
	
	To proceed, we will consider the cases of $A_4 \ne 0$ and $A_4 = 0$ separately.
	
	If $A_4 \ne 0$, then, by adding a suitable multiple of $\omega^3$ to $\omega^4$, one can 
	arrange that $A_{6,3} = 0$. This exhausts the $G_4$-action and results in an $e$-structure.
	Thus, we write
	\[
		\phi = A_{7+j}\omega^j.
	\]
	Since $A_{10}$ is eliminated by $\phi\w\omega^3$, there are
	$10$ invariant functions $A_\rho$ $(\rho = 1,2,\ldots,9, 11)$. For $A_7, A_8, A_9, A_{11}$, define 
	$A_{\rho,j}$ by $\ed A_{\rho} = A_{\rho,j}\omega^j$. Now $\ed^2\omega^4 = 0$ 
	yields $6$ polynomial equations ($\ed^2\omega^i = 0$ for $i = 0,\ldots, 3$ are identities),
	which are solved for $6$ of the 
	$A_{\rho,j}$'s $(\rho = 7,8,9,11)$. Regarding the Pfaffian system generated by the ten
	 $\ed A_{\rho} - A_{\rho,j}\omega^j$, the tableau is $25$-dimensional; absorbability
	of the torsion implies that
	\[
		A_{11} = \frac{1}{3A_4} (18A_4^2A_6+6A_3A_4+9A_5+5).
	\]
	Assigning this relation, $\ed^2\omega^i = 0$ $(i = 0,\ldots, 4)$ remain identities; 
	regarding the Pfaffian system generated by $\ed A_\rho - A_{\rho,j}\omega^k$ for 
	$\rho = 1,2,\ldots, 9$, the tableau is $19$-dimensional, and the torsion is absorbable only 
	when $A_{6,2}$ is a 
	specific fourth-degree polynomial in $A_2,\ldots, A_7$. The updated tableau 
	is $18$-dimensional, and the torsion is absorbable only when $A_{5,2}, A_{6,1}$ are 
	specific functions of the other variables. This reduces the dimension of the tableau to $16$;
	the new torsion can be absorbed. One can verify that the tableau 
	has the Cartan characters 
	\begin{equation}\label{CaseIIaCartanChar_true}
		(s_1,s_2,s_3, s_4, s_5) = (9,6,1,0,0), 
	\end{equation}
	and the dimension of its first prolongation is $24$. Thus, the tableau is involutive.
	By Cartan's third fundamental theorem, the germs of 
	hyperbolic Monge--Amp\`ere systems in Case IIa satisfying $A_4\ne 0$ depend
	on precisely $1$ arbitrary function of $3$ variables.
	
	If $A_4 = 0$, then from the vanishing of $\ed A_4$ we obtain $A_2 = 0$ and $A_5 = -1/3$;
	and from $\ed A_2$ and $\ed A_5$ we obtain 
	$A_{6,2} = (7-A_3)/3$. Computing $\ed^2A_6$ and reducing modulo $\omega^0,\omega^1,\omega^2$, we obtain
	\begin{equation}\label{CaseIIa_A63}
		\ed A_{6,3} \equiv -\frac{2}{3}\omega^4 \mod\omega^0,\omega^1,\omega^2,\omega^3.
	\end{equation}
	Since $A_{6,3}$ is an invariant, we shall rename it as $A_7$; defining $A_{7,j}$ by 
	$\ed A_7 = A_{7,j}\omega^j$, 
	\eqref{CaseIIa_A63} implies that, by adding a multiple of 
	$\omega^3$ into $\omega^4$, we can arrange that $A_{7,3} = 0$. This yields an $e$-structure,
	and we can write
	\[
		\phi = A_{8+j}\omega^j.
	\]
	Thus, the primary invariants are $A_1, A_3, A_6, A_7, A_8, A_9, A_{10}, A_{12}$.
	Proceeding as in the case of $A_4\ne 0$, we find, after a somewhat lengthy computation,
	that, in order for the structure equations to be involutive, we must have
	\[
		A_{12} = \frac{1}{2}(15 A_3 - 7),
	\]
	and all $A_{\rho,j}$'s $(\rho = 1,3,6,7,8,9,10)$ are specific functions of the $A_\rho$'s and 
	10 `free' $A_{\rho,j}$'s. However, regarding the Pfaffian system $I$ generated by the seven
	$\ed A_\rho - A_{\rho,j}\omega^j$, the torsion is absorbable while the tableau is not
	involutive---the Cartan characters are $(s_1,s_2,s_3,s_4,s_5)=(7,2,1,0,0)$, and the first prolongation of the tableau is
	$13$ dimensional. Thus, we need to prolong. This amounts to augmenting $I$ by adjoining, for each of the $10$ `free' $A_{\rho,j}$'s, the form $\ed A_{\rho,j} - A_{\rho,jk}\omega^k$. The result is a Pfaffian system $I^{(1)}$. Further computation, which involves annihilating the intrinsic torsion for $I^{(1)}$, 
	  shows that only $11$ of the $A_{\rho,jk}$'s are free; at this stage, the torsion is absorbable,
	  and 
	 the tableau of $I^{(1)}$ is involutive with the Cartan characters
	\[
		(s_1,s_2,s_3,s_4,s_5) = (9,2,0,0,0).
	\]
	Thus, we conclude that the $A_4 = 0$ case has the local generality of $2$ arbitrary functions of $2$ variables.
	
	Combining both cases, the proposition follows.
	\end{proof}

\begin{remark}\label{Cartan3rdRmk}
	We observe that the last Cartan character in \eqref{CaseIIaCartanChar_crude} coincides with
	that in \eqref{CaseIIaCartanChar_true}. Since the latter is obtained via 
	a much lengthier computation, it is natural to ask whether the computation of \eqref{CaseIIaCartanChar_crude} already implies the generality. To this end, it would be desirable to have
	a version of Cartan's
	third fundamental theorem that can be applied to $G$-structure
	equations with $G\ne \{e\}$ and yield generality results. As of writing, 
 	the author isn't aware of such a theorem.
\end{remark}

\subsection{Case IIb. $P_{01} = 0$} Similar to \eqref{CaseIIaP02} and
\eqref{CaseIIaP10}, computing
\[
\begin{aligned}
	\ed^2\omega^1\mod\omega^2,\omega^4,\qquad& \ed^2\omega^2\mod\omega^1,\omega^4,\\
	\ed^2\omega^3\mod\omega^2,\omega^4,\qquad&\ed^2\omega^4\mod\omega^2,\omega^3
\end{aligned}
\] 
yields
\[
		P_{02} = P_{30}=0.
\]
Now
\[
	\begin{aligned}
		\frac{1}{3}\ed^2\omega^0&\equiv -(3P_{12}P_{13}+P_{11})\,\omega^0\w\omega^2\w\omega^3 \\
			&\qquad\qquad
				+\ed P_{12}\w\omega^0\w\omega^2+\ed P_{13}\w\omega^0\w\omega^3
				\mod\omega^1,\omega^4,\\
		\ed^2\omega^1& \equiv -(3P_{12}P_{13}+ P_{11}-P_{32})\,\omega^1\w\omega^2\w\omega^3
				\\
			&\qquad\qquad
				+\ed P_{12}\w\omega^1\w\omega^2+\ed P_{13}\w\omega^1\w\omega^3
				\mod\omega^0,\omega^4
	\end{aligned}
\]
implies that
\[
	P_{32}=0.
\]
Since $P_{01}, P_{20}, P_{30}, P_{32}$ are all zero, Lemma~\ref{CaseII5torLemma} implies that
$P_{33}$ must be nonzero; thus, by \eqref{CaseII5torTrans}, 
we will assume that $P_{33}$ has been normalized to
\[
	P_{33} = \lambda\in \{\pm 1\}.
\]
Now 
\[
	\ed^2\omega^2\equiv 2P_{10}\,\omega^0\w\omega^1\w\omega^3\mod\omega^2,\omega^4
\]
shows that 
\[
	P_{10}=0;
\]
moreover, we have
\[
	\begin{aligned}
		\frac{1}{3}\ed^2\omega^0&\equiv -\left(2P_{13}P_{11}+\lambda P_{12}-\frac{2}{3}\right)
			\omega^0\w\omega^1\w\omega^3\\
				&\qquad\qquad
				+\ed P_{11}\w\omega^0\w\omega^1 + \ed P_{13}\w\omega^0\w\omega^3
				\mod\omega^2,\omega^4,\\
		-\frac{1}{2}\ed^2\omega^2&\equiv -\left(2P_{13}P_{11}+\frac{1}{2}\right)
			\omega^1\w\omega^2\w\omega^3\\
				&\qquad\qquad
				+\ed P_{11}\w\omega^1\w\omega^2 - \ed P_{13}\w\omega^2\w\omega^3
				\mod\omega^0,\omega^4,
	\end{aligned}
\]
which implies
\[
	P_{12} = \frac{7}{6\lambda}.
\]

At this stage, the only invariants remaining are $P_{11}$ and $P_{13}$, and $\lambda$ serves
as a discrete invariant, equalling either $1$ or $-1$.

\subsubsection{The structure equations}
For $P_{11}$ and $P_{13}$, define $P_{11k}, P_{13k}$ $(k = 0,1,\ldots, 4)$ by
		\[
			\ed P_{ij} = P_{ijk}\omega^k \quad \mbox{for }(i,j) = (1,1), (1,3).
		\]
	Applying $\ed^2 = 0$ to $\omega^0,\ldots, \omega^4$ and reducing 
the equation of $\ed^2\omega^4$ modulo $\omega^3$, the resulting polynomial equations
can be solved for the $P_{ijk}$'s, with the result:
	\[
		\begin{aligned}
			\ed P_{11} &= P_{111}\omega^1 - \frac{7}{6\lambda}P_{11}\omega^2
					+\left(2P_{13}P_{11}+P_{131} + \frac{1}{2}\right)\omega^3,\\
			\ed P_{13} & = -\frac{7}{6\lambda}\omega^0 + P_{131}\omega^1
				 - \frac{2P_{11}\lambda +7P_{13}}{2\lambda}\omega^2 + P_{133}\omega^3.
		\end{aligned}
	\]

	By renaming $P_{11}$ and $P_{13}$ as $A_1$ and $A_2$, respectively, and $\phi_7$ as $\phi$, we record the structure equations obtained at this stage
	in \eqref{Apdx:CaseIIbStrEqn_om} and 
	\eqref{Apdx:CaseIIbStrEqn_As} of Appendix~\ref{Apdx:StrEqn}.

	Incidentally, the Pfaffian system generated by $\ed A_1 - A_{1,k}\omega^k$
	and $\ed A_2 - A_{2,k}\omega^k$ has the tableau
	\begin{equation}\label{CaseIIb_tableau_crude}
		\left(\begin{array}{ccccc}
			0& \pi_1 &0& \pi_2&0\\
			0& \pi_2 & 0& \pi_3 & 0
		\end{array}\right),
	\end{equation}
	where $(\pi_1,\pi_2,\pi_3) = (\ed A_{1,1}, \ed A_{2,1},\ed A_{2,3})$,
	and the torsion can be completely absorbed. It is easy to see that \eqref{CaseIIb_tableau_crude} is involutive 
	with the Cartan characters
	$(s_1,s_2,s_3,s_4,s_5) = (2,1,0,0,0).$ However, as in Case IIa, we do not have an 
	$e$-structure yet, and more work is needed to determine the generality. (As in Case IIa, 
	it is interesting to note that the last Cartan character for \eqref{CaseIIb_tableau_crude}
	is consistent with the conclusion of Proposition~\ref{CaseIIbGen} below. Whether this is a
	coincidence deserves further investigation. See also Remark~\ref{Cartan3rdRmk}.)

\subsubsection{The generality}

\begin{prop}\label{CaseIIbGen}
	Locally the hyperbolic Monge--Amp\`ere systems in Case IIb
	can be uniquely determined (up to contact transformations) by specifying $1$ arbitrary function of $2$ variables.
\end{prop}
\begin{proof}
	Using \eqref{Apdx:CaseIIbStrEqn_om} and \eqref{Apdx:CaseIIbStrEqn_As}, 
	computing $\ed^2 A_2$ yields
	\begin{equation}\label{CaseIIb_dA23}
		\ed^2 A_2 \equiv -\omega^3\w\left(\frac{7}{6\lambda}\omega^4 + \ed A_{2,3}\right)
		\mod\omega^0, \omega^1,\omega^2.
	\end{equation}
	Since $A_{2,3}$ is an invariant, we rename it as $A_3$, and define $A_{3,j}$ by
	$\ed A_3 = A_{3,j}\omega^j$. By \eqref{CaseIIb_dA23}, $A_{3,4}\ne0$, and thus
	one can add a multiple of $\omega^3$ into $\omega^4$ to arrange that 
	$A_{3,3}=0$. This yields an $e$-structure, and we can write
	\[
		\phi = A_{4+j}\omega^j.
	\]
	Noting that $A_7$ is eliminated by $\phi\w\omega^3$, 
	apart from $\lambda$, the invariants are now $A_{\rho}$ for $\rho = 1,\ldots, 6,8$.
	Computing $\ed^2\omega^4$ now yields $6$ polynomial equations, which can be solved
	for $6$ of the $A_{\rho,j}$'s (in our computation, $A_{4,1}$, $A_{4,2}$, $A_{4,4}$, 
	$A_{5,2}$, $A_{5,4}$, $A_{6,4}$); $\ed^2 = 0$ applied to $\omega^i$ $(i = 0,\ldots, 3)$
	yields identities, as expected. 
	
	Now consider the Pfaffian system $I$ generated by the seven $\ed A_\rho - A_{\rho,j}\omega^j$'s  for $\rho = 1,\ldots, 6,8$. The tableau is $20$-dimensional, and the torsion is absorbable 
	only when
	\[
		A_{3,0} = A_1+\frac{7}{\lambda} A_2,\quad
		A_{3,2} = -\left(2A_{2,1} - \frac{7}{\lambda}A_2^2+\frac{14}{3\lambda}A_3+\frac{1}{2}\right),
		\quad
		A_{3,4}= -\frac{7}{6\lambda}.
	\]
	The updated tableau is $17$-dimensional, and the new torsion is absorbable only when
	\[
		A_8 = \frac{6\lambda}{7}A_1 + 10 A_2
	\]
	and $A_{2,1}, A_{3,1}$ are specific functions of $A_\rho$ and $\lambda$. Using these relations,
	$\ed^2\omega^i = 0$ $(i = 0,\ldots, 4)$ remain identities, and the primary invariants
	(apart from $\lambda$) are $A_\rho$ for $\rho = 1,\ldots,6$. Regarding the Pfaffian
	system generated by the six $\ed A_\rho - A_{\rho,j}\omega^j$, the tableau is
	$10$-dimensional, and the torsion is absorbable when $A_{4,0}, A_{4,3}, A_{5,0}, A_{6,0}, A_{6,3}$ are 
	specific functions of the other variables. The updated tableau is $5$-dimensional and is involutive
	with the Cartan characters
	\[
		(s_1,s_2,s_3,s_4,s_5) = (4,1,0,0,0),
	\]
	while the torsion is absorbable. Thus, by Cartan's third fundamental theorem, 
	the germs of hyperbolic Monge--Amp\`ere systems in Case IIb depend precisely
	on $1$ arbitrary function of $2$ variables.
	\end{proof}

\subsection{Conclusion}
Combining Propositions~\ref{CaseIIaGen} and \ref{CaseIIbGen}, the local generality of 
hyperbolic Monge--Amp\`ere systems in Case II is determined, as we state in the following.

\begin{theorem}\label{CaseII_GenThm}
	Locally the hyperbolic Monge--Amp\`ere systems satisfying 
	$S_1=0$, $S_2\ne 0$ and $Q_1\ne 0$, $Q_2=0$ can be uniquely determined (up to contact transformations) by specifying $1$ arbitrary function of $3$ variables.
\end{theorem}


\section{Case III: $Q_1, Q_2\ne 0$}

 In this case, by applying a $G_1$-action, one can normalize to
\[
	Q_1 = Q_2 = 1.
\]
The result is a sub-bundle $\Gcal_2\subset\Gcal_1$ on which the restriction of $\phi_0$ and 
$\phi_1$ become semi-basic; thus, there exist functions $P_{0j}, P_{1j}$ $(j = 0,1,\ldots, 4)$
such that
\[
	\phi_0 = P_{0j}\omega^j,\qquad \phi_1 = P_{1j}\omega^j.
\]
Substituting these in the structure equation \eqref{S1VanstrEqn}, we compute
\[
\begin{aligned}
	\ed^2\omega^1& \equiv\phantom{+}  (3P_{10} - P_{00})\,\omega^0\w\omega^2\w\omega^3&&\mod\omega^1,\omega^4,\\
	\ed^2\omega^3& \equiv\phantom{+}   (3P_{10} +P_{00})\,\omega^0\w\omega^1\w\omega^4&&\mod\omega^2,\omega^3,\\
	\ed^2\omega^1&\equiv\phantom{+}   (3P_{14} - P_{04})\,\omega^2\w\omega^3\w\omega^4&&\mod\omega^0,\omega^1,\\
	\ed^2\omega^3&\equiv -(3P_{12}+P_{02})\,\omega^1\w\omega^2\w\omega^4&&\mod\omega^0,\omega^3,
\end{aligned}
\]
which implies that
\[
	P_{00} = P_{10} = 0, \quad P_{04} = 3P_{14}, \quad P_{02} = -3P_{12}.
\]

Furthermore, computing $\ed^2\omega^0$ and reducing modulo 
$\{\omega^2,\omega^3,\omega^4\}$ and $\{\omega^1,\omega^2,\omega^4\}$, respectively,
we obtain
\begin{equation}\label{CaseIII_P01P03}
\left.
\begin{aligned}
	\ed P_{01}&\equiv -3\phi_3P_{12}\\
	\ed P_{03}&\equiv \phantom{+}3\phi_7 P_{14}
\end{aligned}
\right\}\mod\omega^0,\omega^1,\ldots,\omega^4.
\end{equation}
Now, whether further structure reduction is possible depends on whether 
$P_{12}$ and $P_{14}$ are nonzero. Regarding this, we have the following.

\begin{lemma}\label{CaseIII_P12P14NonVan}
	Both $P_{12}$ and $P_{14}$ must be nonzero.
\end{lemma}
\begin{proof}
	We compute
	\[
		\begin{aligned}
			\ed^2\omega^1\equiv& -(1+4P_{12}P_{14})\,\omega^1\w\omega^2\w\omega^4\\
					&
						+\ed P_{12}\w\omega^1\w\omega^2
						+ \ed P_{14}\w\omega^1\w\omega^4&&\mod\omega^3,\\
			-\frac{1}{3}\ed^2\omega^0\equiv&
						\phantom{+}\ed P_{12}\w\omega^0\w\omega^2 - \ed P_{14}\w\omega^0\w\omega^4&&\mod\omega^1,\omega^3.
		\end{aligned}
	\]
	Letting $P_{ijk}$ denote the coefficient of $\omega^k$ in $\ed P_{ij}$, the above 
	implies that
	\[
		P_{124}= -P_{142} = \frac{1}{2}(1+4P_{12}P_{14}).
	\]
	Thus, either $P_{12}$ or $P_{14}$ being zero (on an open subset of $\Gcal_2$) will give a contradiction.
\end{proof}

By Lemma~\ref{CaseIII_P12P14NonVan} and \eqref{CaseIII_P01P03}, one can perform
a further structure reduction by normalizing 
\[
	P_{01} = P_{03} = 0.
\]
The result of this reduction is an $e$-structure; thus, we can write
\[
	\phi_3 = P_{3j}\omega^j, \qquad \phi_7 = P_{7j}\omega^j.
\]
Further relations among the $P_{ij}$'s are obtained via the following computation:
\[
	\begin{aligned}
		\ed^2\omega^1&\equiv (-P_{14}P_{70} + P_{12}-P_{30})\,\omega^0\w\omega^1\w\omega^3+\ed P_{13}\w\omega^1\w\omega^3 
			&&\mod\omega^2,\omega^4,\\
		-\ed^2\omega^2&\equiv (2P_{14}P_{70}+4P_{12} - P_{30})\,\omega^0\w\omega^2\w\omega^3  + \ed P_{13}\w\omega^2\w\omega^3&&\mod\omega^1,\omega^4,
	\end{aligned}
\]
which implies that
\[
	P_{70} = -\frac{P_{12}}{P_{14}};
\]
similarly, computing
\[
	\ed^2\omega^3\mod\omega^2,\omega^4\qquad
	\mbox{and}
	\qquad
	-\ed^2\omega^4\mod\omega^2,\omega^3
\]
yields
\[
	P_{30} = \frac{P_{14}}{P_{12}}.
\]
Computing
\[
	\ed^2\omega^0\equiv (3P_{12}P_{33}+3P_{14}P_{71}+2)\,\omega^0\w\omega^1\w\omega^3
	\mod\omega^2,\omega^4
\]
shows that there exists a new function $R$ such that
\[
		 P_{33} = \frac{R - 1}{3P_{12}}, \qquad P_{71} = \frac{-R-1}{3P_{14}}.
\]

At this stage, there are $9$ primary invariants
\begin{equation}\label{CaseIII_PrimInv}
	P_{11}, P_{12}, P_{13}, P_{14}, P_{32}, P_{34}, P_{74}, P_{72}, R,
\end{equation}
and we shall rename them, in the exact order as they appear in \eqref{CaseIII_PrimInv}, as
\[
	A_\sigma \quad (\sigma = 1,2,\ldots, 9).
\]
The structure equations $\ed\omega^i$ $(i = 0,1,\ldots, 4)$ 
now take the form of \eqref{Apdx:CaseIII_StrEqn_om} in Appendix~\ref{Apdx:StrEqn},
and we are ready to prove the generality result in this case.

\begin{theorem}\label{CaseIII_GenThm}
	Locally the hyperbolic Monge--Amp\`ere systems satisfying 
	$S_1=0$, $S_2\ne 0$ and $Q_1, Q_2\ne 0$ can be uniquely determined (up to contact transformations) by specifying $2$ arbitrary function of $3$ variables.
\end{theorem}
\begin{proof}
	For each $A_\sigma$ $(\sigma =1,2,\ldots, 9)$, define $A_{\sigma,j}$ $(j = 0,1,\ldots, 4)$ by
\[
	 \ed A_\sigma = A_{\sigma,j} \omega^j.
\]
Applying $\ed^2=0$ to the equations \eqref{Apdx:CaseIII_StrEqn_om} yields a system of $33$ distinct polynomial equations in the $A_\sigma$ and $A_{\sigma,j}$. The solution of this system 
expresses $25$ of the $A_{\sigma,j}$ in terms of the other variables.

Next we examine the Pfaffian system $I$ generated by the nine $\ed A_\sigma - A_{\sigma,j}\omega^j$
$(\sigma = 1,2,\ldots, 9)$. The tableau of $I$ is $20$-dimensional, and the torsion is 
absorbable exactly when two more relations are satisfied (in our calculation, $A_{6,2}$
and $A_{8,4}$ are solved in terms of other variables). Updating these relations, 
the new tableau is $18$-dimensional, and the new torsion can be completely absorbed.

Now a computation shows that the tableau of $I$ has the Cartan characters
\[
	(s_1,s_2,s_3,s_4,s_5) = (9,7,2,0,0),
\]
and its first prolongation is $29$ dimensional. By Cartan's test, the Pfaffian system is involutive, and, by Cartan's third fundamental theorem, the underlying structure locally depends on $2$ arbitrary functions of $3$ variables.
\end{proof}

The final structure equations in this case are recorded in \eqref{Apdx:CaseIII_StrEqn_om} to \eqref{Apdx:CaseIII_StrEqn_A9} in Appendix~\ref{Apdx:StrEqn}.

\section{Analysis of symmetry}

In this section, we classify those $S_1 = 0$ hyperbolic Monge--Amp\`ere systems that 
have cohomogeneity at most one. Here `cohomogeneity' is interpreted infinitesimally: 
a hyperbolic Monge--Amp\`ere system $(M,\Ecal)$ is of \emph{cohomogeneity $k$} near $p\in M$ if, for sufficiently small neighborhoods $U\ni p$,
the Lie algebra of infinitesimal symmetry $\mathfrak{l}\subset\mathfrak{X}(U)$ of $(U, \Ecal|_U)$ satisfies
$\dim(\mathfrak{l}_q) = 5 - k$ for each $q\in U$, where $\mathfrak{l}_q = \span\{X_q: X\in \mathfrak{l}\}$. A system of cohomogeneity zero is called \emph{homogeneous}.

The main result of this section, an immediate consequence of Propositions~\ref{CaseIICohomog_prop} and \ref{CaseIIICohomog_prop} below,
is the following.
\begin{theorem}\label{CohomogThm}
Any hyperbolic Monge--Amp\`ere system satisfying $S_1 =0$, $S_2\ne 0$ and having cohomogeneity $k\le 1$ must 
satisfy $Q_1 = Q_2 = 0$; thus, up to contact transformations, it corresponds to a linear PDE of the form \eqref{Q1Q2VanCoordForm}.
\end{theorem}

Now we proceed to examine the cases I, II and III separately.

\begin{prop}\label{CaseIIa_sym_prop}
	The cohomogeneity of a hyperbolic Monge--Amp\`ere system satisfying $S_1 = 0$, $S_2\ne 0$ 
	and $Q_1=Q_2=0$ is equal to the 
	cohomogeneity of its associated 
	geometry $(\Sigma,\Omega; \foli_1,\foli_2)$ (see Section~\ref{Case_I_asociGeom}), which is at most $2$.
\end{prop}
\begin{proof}
	Let $\pi:(M,\Ecal)\to(\Sigma,\Omega;\foli_1,\foli_2)$ be the canonical local submersion, and
	let $X\in \mathfrak{X}(U)$ ($U\subset M$ small open subset) be an infinitesimal symmetry
	of $(M,\Ecal)$. By \eqref{Q1Q2VanStrEqn}, it is easy to see that the flow $F^t$ of $X$ 
	preserves the fibers of $\pi$ and induces an infinitesimal symmetry $\pi_*X$ of $(\Sigma,\Omega;\foli_1,\foli_2)$ on $\pi(U)$. Now, if $\dim(\mathfrak{l}_q) = 5-k$ for $q\in U$, then 
	it is clear that $\span\{\pi_*X: X\in \mathfrak{l}\}$ has dimension at least $2 - k$
	for points in $\pi(U)$. This proves that the cohomogeneity of $(\Sigma,\Omega;\foli_1,\foli_2)$
	is at most that of $(M,\Ecal)$.
	
	Thus, to complete the proof, it suffices to show that, by shrinking the domain if needed, any local symmetry of 
	$(\Sigma,\Omega;\foli_1,\foli_2)$ `lifts' to a local symmetry of $(M,\Ecal)$. To this end, it is 
	convenient to work with the structure bundles. 
	
	Suppose that $\psi:U\to V$ is a local symmetry of $(\Sigma,\Omega;\foli_1,\foli_2)$, where $U,V\subset\Sigma$ are open subsets, and let $\pi_1: \Gcal_{1}\to \Wcal$ and $\pi_2:\Wcal\to \Sigma$ be the canonical local submersions. Let $\Wcal_U = \pi_2^{-1}(U)$, $\Gcal_{1,U} = \pi_1^{-1}(\Wcal_U)$, and similarly define $\Wcal_{V}$ and $\Gcal_{1,V}$. Let $\omega_U^k$
	$(k = 0,1,\ldots, 4)$ be the tautological $1$-forms on $\Gcal_{1,U}$, which satisfy \eqref{Q1Q2VanStrEqn} and \eqref{Q1Q2VanStrEqnConn}, and similarly for $\omega^k_V$. 
	Note that $\omega^1_U, \omega^3_U$ and $\phi_{1,U}$ are pull-backs of $1$-forms
	defined on $\Wcal_U$, which we denote by the same symbol, and similarly for $\omega^1_V,\omega^3_V$ and $\phi_{1,V}$.
		\[
		\begin{tikzcd}[column sep=normal]
				\Gcal_{1,U} \arrow[r,dashed,"\Psi"] \arrow[d, "\pi_{1}",swap]  &   \Gcal_{1,V}\arrow[d,"\pi_{1}"]& 		\\
				\Wcal_U  \arrow[r,"\tilde \psi"]\arrow[d,"\pi_{2}",swap]&  \Wcal_V\arrow[d,"\pi_{2}"]\\
				U\arrow[r, "\psi"]& V
 			\end{tikzcd}
	\]
	
	Since $\psi:U\to V$ is a local symmetry, by shrinking $U$ and $V$, if needed, 
	there exists a 
	diffeomorphism $\tilde \psi: \Wcal_U\to \Wcal_V$ such that $\pi_2\circ\tilde\psi = \psi\circ\pi_2$ and
	$\tilde \psi^*(\omega^1_V,\omega^3_V) = (\omega^1_U,\omega^3_U)$ (see \cite[p.11]{Gardner89}). By \eqref{Q1Q2VanStrEqn}, we have $\tilde\psi^*(\phi_{1,V}) = \phi_{1,U}$. Now let
	$\Lcal \subset \Gcal_{1,U}\times \Gcal_{1,V}$ be the locus defined by 
	$\tilde \psi(\pi_1(u))) = \pi_1(v)$. On $\Lcal$, the restriction of $\omega^\ell_U - \omega^\ell_V$ $(\ell = 1,3)$ and $\phi_{1,U} -\phi_{1,V}$ are identically zero. 
	Let $I$ be the rank-$4$ Pfaffian system, defined on $\Lcal$,
	generated by  
	$\theta^k := \omega^k_U - \omega^k_V$ for $k = 0,2,4$,
	and
	$\eta = \phi_{0,U} - \phi_{0,V}$. 
	
	Using \eqref{Q1Q2VanStrEqn}, we obtain
	\[
		\left.
		\begin{aligned}
			\ed\theta^0&\equiv 0\\
			\ed\theta^2&\equiv - (\phi_{3,U} - \phi_{3,V})\w\omega^1
							\\
			\ed\theta^4& \equiv -(\phi_{7,U} - \phi_{7,V})\w\omega^3
							\\
			\ed\eta & \equiv 0
		\end{aligned}
		~
		\right\}\mod I.
	\]
	By Cartan's test, $I$ is involutive; thus, there exists a local diffeomorphism 
	$\Psi: \Gcal_{1,U} \to \Gcal_{1,V}$, whose graph is contained in $\Lcal$, that satisfies $\Psi^*\omega^k_V = \omega^k_U$
	for $k = 0,1,\ldots, 4$.  It follows (see \cite[p.11]{Gardner89}) that $\Psi$ descends to a local symmetry of 
	$(M,\Ecal)$ that lifts $\psi$. 
\end{proof}

Given a hyperbolic Monge--Amp\`ere
PDE of the form \eqref{Q1Q2VanCoordForm}, the cohomogeneity of the corresponding 
Monge--Amp\`ere system
can be computed explicitly by using the function $f(x,y)$, as the following proposition shows.

\begin{prop}\label{CaseIsym_via_f}
	Given a PDE of the form \eqref{Q1Q2VanCoordForm}, its corresponding Monge--Amp\`ere system $(M,\Ecal)$
	is homogeneous if and only if $A(x,y)=-2\e^{-2f}f_{xy}$ is constant. For $p\in M$ satisfying
	$\ed A|_p\ne 0$, $(M,\Ecal)$ is of cohomogeneity $1$ near $p$
	if and only if, in a neighborhood of $p$, $A_x$ and $A_y$ are not identically zero, and the following exterior derivatives are 
	scalar multiples of 
	$\ed A$
	\[
		\ed\left(\e^{-2f} A_xA_y\right),\quad
		\ed\left(\e^{-2f}A_{xy}\right).
	\]
\end{prop}
\begin{proof} By Proposition~\ref{CaseIIa_sym_prop}, it suffices 
to work with the associated geometry $(\Sigma,\Omega;\foli_1,\foli_2)$,
which we denote by $\Sigma$ for simplicity; let $\Wcal$ be its structure bundle.

First suppose that $\Sigma$ is homogeneous; clearly $A$ must be locally constant. Conversely, 
if $A$ is constant, then, locally $\Wcal$ has a Lie group structure with 
$\omega^1,\omega^3,\phi_0$ being left-invariant $1$-forms; thus, $\Sigma$ is homogeneous.
This proves the first statement.

Now we prove the second statement. By
Section~\ref{CaseI_intg_Section},  
there exist local coordinates $(x,y,t)$ on $\Wcal$ and a function $f(x,y)$
such that
	\[
		\omega^1 = \e^{f+t} \ed x , \quad
		\omega^3 = \e^{f - t} \ed y,\quad
		\phi_1 = -\ed t + f_x \ed x - f_y\ed y,
	\]
	and $A = -2\e^{-2f} f_{xy}$.
(The function $f$ is precisely the one occurring in \eqref{Q1Q2VanCoordForm}.)
	Writing $\ed A = A_1\omega^1+A_3\omega^3$, we have
	\begin{equation}\label{CaseI_sym_A1A3}
		A_1 = A_x \e^{-f - t}, \quad A_3 = A_y \e^{-f + t}.
	\end{equation}
	Suppose that $\ed A|_{(x,y)}\ne 0$ at $(x,y)\in \Sigma$.
	By applying the transformation $(\omega^1,\omega^3,\phi_1)\mapsto(\omega^3,-\omega^1,-\phi_1)$, if needed, we can assume that $A_3 > 0$; then, by setting $A_3 = 1$, which corresponds to choosing $t = \ln\left(\e^f/A_y\right)$,
	we obtain a coframing ($e$-structure)
	on $\Sigma$.
 	Now $A_1 = \e^{-2f} A_xA_y$ is an invariant; moreover, 
	$\phi_1 = H_1\omega^1+H_3\omega^3$, where
	\[
		H_1 = \e^{-2f} A_{xy}, \quad H_3 = \frac{1}{A_y^2} (A_{yy} - 2f_yA_y)
	\]
	are also invariants.
	
	If $\Sigma$ has cohomogeneity $1$, then it is clear that
	$\ed A_1$, $\ed H_1$, $\ed H_3$ must be multiples of $\ed A$.
	It remains to show that $A_x$ cannot vanish identically. Otherwise, $A_x\equiv 0$, and
	$A$ would be a function of $y$ alone; thus, the same must hold for $H_3$. From the expression of 
	$H_3$, we see that $f_y$ must be a function of $y$ alone, so $f_{xy}=0$, and 
	$A\equiv 0$. This implies that $\Sigma$ is homogeneous, a contradiction.
	
	Conversely, if $\ed A_1$ and $\ed H_1$ are multiples of $\ed A$, then
	in a neighborhood of $(x,y)$, there exist functions $\psi_1,\psi_2$ such that
	\[
		A_1 = \psi_1(A), \qquad H_1 = \psi_2(A).
	\]
	Thus, $\ed A_1 = \psi_1'(A) (A_1\omega^1+\omega^3)$; on the other hand,
	computing $\ed(\ed A) = \ed (A_1\omega^1+\omega^3)$ shows that the coefficient of $\omega^3$ in $\ed A_1$ is $A_1H_3 + H_1$. Since $A_x$ is not identically zero, the same holds for $A_1$;
	from this and $\psi_1'(A) = A_1H_3+H_1$ we see that $H_3$ is a function of $A$. Since all the invariants
	and their $\omega$-derivatives are functions of $A$, by a theorem of Cartan (in \cite{Cartan37}), 
	the algebra infinitesimal symmetry $\mathfrak{l}$ of $\Sigma$
	satisfies $\dim(\mathfrak{l}_{(x',y')}) = 1$ for all $(x',y')$ in a neighborhood of $(x,y)$.
	This completes the proof.
\end{proof}

\begin{example}
A class of cohomogeneity-$1$ examples is obtained by putting $f(x,y) = \varphi(xy)$, 
where $\varphi(u)$ is an arbitrary function of $1$ variable such that
$\e^{-2\varphi}(\varphi' + u\varphi'')$ is non-constant. 
It will be interesting to know whether these are the only cohomogeneity-$1$ examples.
\end{example}

\begin{prop}\label{CaseIICohomog_prop}
	A hyperbolic Monge--Amp\`ere system with $S_1 = 0$, $Q_1\ne 0$ and 
	$Q_2 = 0$ must have cohomogeneity at least $2$.
\end{prop}
\begin{proof}
	 We begin with {Case IIa}. In \eqref{Apdx:CaseIIaStrEqn_om}, let $C^\rho_i$ ($\rho = 1,2,\ldots, 6$; $i = 0,1,\ldots, 4$) denote coefficient of $\omega^i$ in $\ed A_\rho$. If the underlying
	 structure has cohomogeneity at most one, then all the $2\times 2$ minors of 
	 $(C^\rho_i)$ must vanish. A computation shows that this condition yields two possibilities:
	 \begin{enumerate}[\bf i.]
	 	\item $(A_1, A_2, A_4, A_5) = \left(-\frac{2}{3}, 0, 0, -\frac{2}{3}\right)$
			and the only nonzero column of $(C^\rho_i)$ is 
			\[
				(C^\rho_4) = \left(-\frac{2}{3} - A_3,\; 0, \;A_{3,3}, \;\frac{1}{3},\; -\frac{7}{3}, \;A_{6,3}\right)^T.
			\]
			Note that $A_4 = 0$ implies that $A_{3,3}$ is an invariant;
			thus, computing $\ed^2A_3$ yields
			\[
				\ed A_{3,3} \equiv \omega^0\mod\omega^1,\ldots\omega^4.
			\]
			Clearly, the rank of $\span\{\ed A_1,\ldots, \ed A_6, \ed A_{3,3}\}$ 
			is at least two.
		\item $(A_1, A_2, A_3, A_4, A_5) = \left(-\frac{1}{3}, 0, -2, 0, -\frac{1}{3}\right)$
			and the only nonzero row of $(C^\rho_i)$ is 
			\[
				(C^6_i) = \left(-\frac{2}{3}, \; A_{6,1}, \; 3, \; A_{6,3}, \; 0\right).
			\]
			The vanishing of $A_4$ implies that $A_{6,3}$ is an invariant; thus,
			computing $\ed^2A_6$ yields
			\[
				\ed A_{6,3}\equiv -\frac{2}{3}\omega^4 \mod\omega^0,\ldots,\omega^3.
			\]
			Clearly, the rank of $\span\{\ed A_1,\ldots, \ed A_6, \ed A_{6,3}\}$ is at least two.
	 \end{enumerate}
	 This proves the proposition for systems in Case IIa.
	 
	 Now we turn to Case IIb. Since $\ed A_2$ contains no $\omega^4$ term,
	 $A_{2,3}$ is an invariant. Computing $\ed^2 A_2$ using \eqref{Apdx:CaseIIbStrEqn_om} and \eqref{Apdx:CaseIIbStrEqn_As}, we find 
	 \[
	 	\ed A_{2,3}\equiv -\frac{7}{6\lambda}\omega^4 \mod \omega^0,\ldots,\omega^3.
	 \]
	 Comparing this relation with \eqref{Apdx:CaseIIbStrEqn_As} shows that 
	 $\span\{\ed A_1, \ed A_2, \ed A_{2,3}\}$ has rank at least $2$.
	This verifies the proposition for systems in Case IIb.
\end{proof}

\begin{prop}\label{CaseIIICohomog_prop}
	A hyperbolic Monge--Amp\`ere system with $S_1 = 0$, $Q_1, Q_2\ne 0$
   	 must have cohomogeneity at least $2$.
\end{prop}
\begin{proof}
	Since we have an $e$-structure in this case, we will assume that the exterior derivatives
	of all invariants are scalar
	multiples of each other, which is necessitated by the condition of having `cohomogeneity at most one', and then we derive a contradiction.
	
	In \eqref{Apdx:CaseIII_StrEqn_A1}-\eqref{Apdx:CaseIII_StrEqn_A9}, let $C^\sigma_i$ denote the coefficient of 
	$\omega^i$ in $\ed A_\sigma$, and let $M^{\sigma_1\sigma_2}_{i_1i_2}$
	denote the determinant of the sub-matrix $(C^\sigma_i)$ formed by $\sigma \in \{\sigma_1,\sigma_2\}$ and $i \in\{i_1, i_2\}$. We compute
	\[
	\begin{aligned}
		M^{23}_{04} &= \frac{1}{2A_2}(4A_2A_4+1)(2A_2^2 - A_4), \\
		M^{14}_{02}& = \frac{-1}{2A_4} (4A_2A_4+1)(2A_4^2+A_2).
	\end{aligned}
	\]
	Thus, in order for $\rank(C^\sigma_i)\le 1$, one of the following two possibilities must hold:
	\begin{enumerate}[\bf i.]
		\item $A_2 = - A_4 = -\frac{1}{2}$. In this case, the vanishing of $\ed A_2, \ed A_4$
			implies that $-A_5 = A_6 = A_1$ and $A_7 = -A_8 = A_3$. Now computing 
			\[
				M^{19}_{04} = -\frac{3}{2}(A_1+A_3)(A_1 - 2A_3),
				\quad
				M^{39}_{04} = \frac{3}{2}(A_1+A_3)(-A_3+2A_1)
			\]
			shows that $A_3 = -A_1$. Meanwhile, the row $(C^9_i)$ reads
			\[
				\left(0,\; A_{9,1}, \; -A_9 - \frac{5}{2}-3A_1^2-\frac{3}{2}A_{1,1}, 
				\; A_{9,3}, \; -A_9 + \frac{5}{2} - 3A_1^2 - \frac{3}{2}A_{1,3}\right).
			\]
			By computing $\ed^2 A_9$, we find
			\[
				\ed A_{9,1}\equiv \frac{1}{2}\left(10+3(A_{1,1} - A_{1,3})\right)\omega^0
						\mod\omega^1,\ldots,\omega^4.
			\]
			It is easy to see that if $A_1=0$, then $\rank\{\ed A_9, \ed A_{9,1}\} = 2$.
			Now assuming $A_1\ne 0$, the condition $M^{19}_{24}=0$ implies that
		 	$A_{1,3} = (10+3A_{1,1})/3$; then from 
			the vanishing of $M^{19}_{12}$ and $M^{19}_{23}$
			one can solve for $A_{9,1}$ and $A_{9,3}$. At this stage,
			computing $\ed^2 A_1$ and reducing modulo $\{\omega^1,\omega^2\}$
			and $\{\omega^2,\omega^3\}$, respectively, we find that the
			coefficient of $\omega^4$ in $\ed A_{1,1}$ must equal
			$-(20/3)-2A_{1,1}$ and $-(10/3) - 2A_{1,1}$ simultaneously, which is impossible.
			
		\item $4A_2A_4+1 = 0$ but $A_2\ne -\frac{1}{2}$.
			In this case, the vanishing of $\ed(A_2A_4)$ yields
			$A_5 = -4A_2^2A_6$ and $A_7 = -A_8/(4A_2^2)$. Now
			\[
				M^{12}_{01} = \frac{1}{2}(1+8A_2^3)(A_1 - 4A_2^2A_6),
				\quad
				M^{34}_{03} = \frac{1}{64A_2^5}(1+8A_2^3)(A_8+4A_3A_2^2),
			\]
			the vanishing of which implies that
			$A_6 = A_1/(4A_2^2)$ and $A_8 = -4A_3A_2^2$. Using these, we compute
			\[
				M^{13}_{02} = \frac{1}{4A_2}(1+8A_2^3)(2A_2A_3 - A_1),
			\]
			the vanishing of which enforces $A_3 = A_1/(2A_2)$.
			From $\rank\{\ed A_1, \ed A_9\}\le 1$, 
			we can express $A_{1,1}, A_{1,3}, A_{9,1}, A_{9,3}$ in terms of 
			$A_1, A_2$ and $A_9$. Then we compute
			$\ed^2A_1$ and find
			that the coefficient of $\omega^1\w\omega^3$ is $4A_1/3$, 
			which implies $A_1 =0$. However, $\ed A_1 = 0$ implies 
			that $A_2 = -1/2$, violating the assumption of this case.
	\end{enumerate}
	To conclude, the hyperbolic Monge--Amp\`ere systems in Case III 
	must have cohomogeneity at least $2$.
\end{proof}

\section{Acknowledgement}
The author would like to thank Professor Robert L. Bryant for helpful discussions on the equivalence method.
Thanks also to NSFC for its support via grant 12401315, and to Shanghai Jiao Tong University for its support via grant WH220407115.

\begin{appendix}

\section{Proof of Proposition~\ref{S2dgnrProp}}\label{Apdx:S2VanPf}

In this appendix, we outline the computation that verifies that $\det(S_2)$ must vanish 
when $S_1 = 0$. By \eqref{GActIdComp} and 
\eqref{GActJ}, the sign of $\det(S_2)$ is an invariant. We show that
both the assumptions $\det(S_2)>0$ and $\det(S_2)<0$ lead to non-involutive structure equations,
and conclude that $\det(S_2) =0$ is the only possibility.

\subsection{Case 1. $\det(S_2)>0$}
By \eqref{GActIdComp}, one can normalize to $S_2 = I_2$; the result of this reduction
is a principal $G_1^+$-bundle $\Gcal_1^+\subset \Gcal$ where $G_1^+\subset G$ is generated by
\[
	h = \diag(1,\Ab,\Ab) \mbox{ where }\Ab\in \SL(2,\R), 
	\quad K = \diag(-1, 1, -1, -1, 1)~ \mbox{ and } J.
\]
By \eqref{GActInfi}, on $\Gcal_1^+$, the restriction of 
\[
	\phi_0, \quad
	\phi_5 - \phi_1, \quad
	\phi_6 - \phi_2,\quad
	\phi_7 - \phi_3
\]
are semi-basic; thus, there exist functions $P_{ij}$ $(i = 0,5,6,7; j = 0,\ldots,4)$ such that
\[
\begin{aligned}
		\phi_0 &=  P_{0j}\omega^j,\quad
		&&\phi_5 = \phi_1 + \sum_{j = 0}^3P_{5j} \omega^j + (P_{54}+P_{63})\omega^4,\\
		\phi_6 &= \phi_2 + P_{6j}\omega^j,\quad
		&&\phi_7 = \phi_3 + \sum_{j = 0}^3 P_{7j}\omega^j + (P_{74} - P_{53})\omega^4;
\end{aligned}
\]
where a special arrangement is made for the $\omega^4$-term in $\phi_5$ and $\phi_7$
so that the structure equations take a simpler form.

On $\Gcal_1^+$, the pseudo-connection form now reads
\[
	\diag(0,\Phi,\Phi), \quad\mbox{where}~\Phi = \left(\begin{array}{cc}\phi_1&\phi_2\\
	\phi_3&-\phi_1\end{array}\right).
\]
By adding suitable semi-basic forms to $\phi_1,\phi_2,\phi_3$ 
\[
\begin{aligned}
	\phi_1&\mapsto \phi_1 + P_{51}\omega^1+P_{52}\omega^2, \\
	\phi_2&\mapsto \phi_2 + P_{52}\omega^1+P_{62}\omega^2,\\
	\phi_3&\mapsto \phi_3  - P_{51}\omega^2+P_{71}\omega^1
\end{aligned}
\]
and adjusting $P_{61}, P_{72}$ in the following way
\[
	P_{61}\mapsto P_{61}+P_{52} , \qquad
	P_{72}\mapsto P_{72} - P_{51}
\]
all terms that involve $P_{51}, P_{52}, P_{62}, P_{71}$ will be absorbed into the pseudo-connection form.
Computing the following expressions
\begin{equation}\label{d2mod12or34}
	\ed^2\omega^i \mod \omega^1,\omega^2 ~(i = 1,2);
	\quad
	\ed^2\omega^j \mod \omega^3,\omega^4 ~(j= 3,4)
\end{equation}
then reveals the following relations
\[
	P_{54} = -P_{04}, \quad P_{61} = -P_{02}, \quad 
	P_{72} = -P_{01}, \quad P_{74} = P_{03}.
\]
Assigning these relations leaves us with $8$ torsion functions:
\[
	P_{0j} ~(j = 0,1,\ldots, 4) \quad\mbox{and } P_{k0} ~(k = 5,6,7).
\]
One may, as usual, find the $G_1^+$-action on these functions by a combination of 
infinitesimal methods and explicit calculation for the actions of $J$ and $K$; however,
at this stage one can already verify that there is no local section of $\Gcal_1^+$ on which
the restriction of the structure equations is involutive. To do so, we simply write
\[
	\phi_i = P_{ij}\omega^j \quad (i = 1,2,3)
\]
and, for the $23$ $P_{ij}$'s, define $P_{ijk}$ $(k = 0,1,\ldots,4)$ by 
\[
	\ed P_{ij} = P_{ijk}\omega^k.
\]
Applying $\ed^2 = 0$ to $\omega^1,\ldots,\omega^4$ yields an incompatible system of polynomial equations.

\subsection{Case 2. $\det(S_2)<0$} Similar to Case 1, here we normalize to 
\[
	S_2 = \left(\begin{array}{cc}0&1\\1&0\end{array}\right).
\]
This reduction results in a bundle $\Gcal_1^-\subset\Gcal$ on which the restriction of 
\[
	\phi_0, \quad \phi_5+\phi_1, \quad \phi_6 - \phi_3, \quad \phi_7 -\phi_2
\]
become semi-basic; thus, we can write
\[
	\begin{aligned}
		\phi_0 & = P_{0j}\omega^j, \quad
		&&
		\phi_5 = -\phi_1 +\sum_{j = 0}^3P_{5j}\omega^j
				+(P_{54}+P_{63})\omega^4,\\
		\phi_6& = \phi_3 + P_{6j}\omega^j, \quad
		&&
		\phi_7 =\phantom{+}\phi_2 + \sum_{j = 0}^3P_{7j}\omega^j 
				+ (P_{74} - P_{53})\omega^4.
	\end{aligned}
\]
By adjusting $\phi_1, \phi_2, \phi_3$ and $P_{62}, P_{71}$, the torsion functions 
$P_{51}, P_{52}, P_{72}, P_{61}$ can be absorbed. Computing the expressions in 
\eqref{d2mod12or34} yields
\[
	P_{54} = P_{71} = 0, \quad P_{62} = -2P_{01}, \quad P_{74} = 2P_{03}.
\]
At this stage, the torsion functions remaining are
\[
	P_{0j} ~(j = 0,1,\ldots, 4) \quad\mbox{and } P_{k0} ~(k = 5,6,7).
\]
Proceeding as in Case 1, one verifies that the structure equations are not involutive.

\section{Proof of Proposition~\ref{P14VanProp}}\label{Apdx:P14VanPf}
In this appendix, we provide details of the computation that shows that if
$P_{14}\ne 0$, then the structure equations \eqref{CaseIIstrEqn} would be non-involutive.

Assuming $P_{14}\ne0$, \eqref{CaseIIstrEqn} implies that one can normalize to 
\[
	P_{13} = 0.
\]
The result is an $e$-structure, and there exist functions $P_{7j}$ $(j = 0,1,\ldots, 4)$ such that
\[
	\phi_7 = P_{7j}\omega^j.
\]
Now, computing 
\[
\begin{aligned}
	\ed^2\omega^1&\equiv [(- P_{30}+P_{12}-P_{14}P_{70})\omega^3 - \ed P_{10}]\w\omega^0\w\omega^1&&\mod\omega^2,\omega^4,\\
	\frac{1}{2}\ed^2\omega^2&\equiv \left[\left(\frac{1}{2}P_{30}+\frac{1}{2}P_{02}+P_{12} - P_{14}P_{70}\right)\omega^3 - \ed P_{10}\right]
	\w\omega^0\w\omega^2&&\mod\omega^1,\omega^4
\end{aligned}
\]
gives two expressions for the $\omega^3$-term in $\ed P_{10}$, which enforces that
\begin{equation}\label{ApdxB:P02}
	P_{02} = -3P_{30}.
\end{equation}
Assigning this relation, there are $12$ $P_{ij}$'s left (note that $P_{73}$ is annihilated in $\phi_7\w\omega^3$); for them, define $P_{ijk}$ $(k = 0,1,\ldots, 4)$ by 
\[
 	\ed P_{ij} = P_{ijk}\omega^k.
\]
Applying $\ed^2 = 0$ to the $\omega^i$'s yields a system of $34$ polynomial equations,
the solution of which expresses $27$ of the $P_{ijk}$'s in terms of the remaining variables
occurring in the system. Among the $P_{ijk}$'s that are solved for, the following are particularly
informative:
\begin{equation}\label{ApdxB:variousPijk}
	\begin{array}{ll}
		P_{010} = \phantom{+}P_{10}P_{01} - 3P_{30}^2,\qquad\qquad
				& P_{303} = -\frac{1}{3}P_{01} - P_{32}, \\[0.25em]
		P_{014}= \phantom{+}P_{14}P_{01}, 
				& P_{324}=  \phantom{+1}P_{14}P_{32},\\[0.25em]
		P_{103} = -P_{14}P_{70} + P_{12} - P_{30},
				& P_{330} = -2P_{10}P_{33} - \frac{4}{3}P_{01} - 2P_{32},\\[0.25em]
		P_{123} = -P_{14}P_{72} + P_{11} - P_{32},
				&	P_{334} = -2P_{14}P_{33}+P_{30},		\\[0.25em]	
		P_{143}=-P_{14}P_{74} - P_{10}.&
	\end{array}
\end{equation}

To proceed, let us remind the reader that we haven't yet specified 
which among the $P_{01}, P_{02}, P_{30}, P_{32}, P_{33}$ has been normalized.
In the following, we will repeatedly use the simple fact: if $P_{ij}$ is constant, then 
the $P_{ijk}$ $(k=0,1,\ldots, 4)$ are zero.

First let us assume that $P_{01} =0$; then the expression of  $P_{010}$ implies that
$P_{30} = 0$; and, in turn, $P_{303}=0$ implies that $P_{32}  = 0$. By \eqref{ApdxB:P02},  
$P_{02}$ is also zero. By Lemma~\ref{CaseII5torLemma}, $P_{33}$ must be nonzero,
and hence we can assume that it has been normalized to a constant. By the expressions of $P_{330}$
and $P_{334}$, we deduce that $P_{10} = P_{14} = 0$. Now $P_{103}=0$ implies that 
$P_{12}=0$; and then $P_{123}=0$ implies that $P_{11}=0$. 

In summary, from $P_{01}=0$ one can, by normalizing $P_{33}$, deduce that $P_{30}$, $P_{32}$, $P_{10}$, $P_{14}$, $P_{12}$, $P_{11}$ are all zero.
Using these, one can compute
\[
	\ed^2\omega^0 \equiv 2\,\omega^0\w\omega^1\w\omega^3\mod\omega^2,
\]
which is impossible.

Thus, the remaining possibility is $P_{01} \ne 0$. By \eqref{CaseII5torTrans}, we may 
assume that $P_{01}$ is normalized to $1$. Thus, $P_{014} = 0$ implies that
$P_{14}=0$, which, via the vanishing of $P_{143}$, implies that $P_{10}=0$. Now $P_{103}=0$ implies
that $P_{12} = P_{30}$.

Computing
\[
	\ed^2\omega^0\equiv -(1+3 P_{11})\,\omega^0\w\omega^2\w\omega^3
		\mod\omega^1
\]
implies that
\[
	P_{11} = -\frac{1}{3}.
\]
Now
\[
	\ed^2\omega^0 = 2\,\omega^0\w\omega^1\w\omega^3,
\]
which is impossible.

In conclusion, the invariant $P_{14}$ must be zero.

\section{The structure equations}\label{Apdx:StrEqn}
In this appendix, we record the structure equations 
in Cases IIa, IIb and III.
\subsection{Case IIa} In \eqref{Apdx:CaseIIaStrEqn_om} below, $\phi$, which occurs in $\ed\omega^4$, is the pseudo-connection form. As \eqref{Apdx:CaseIIaStrEqn_As} shows, all torsion functions $A_\rho$ $(\rho = 1,2\ldots6)$ are constant along group fibers. 
While these structure equations are not for an $e$-structure, further reduction to an $e$-structure is possible (see the proof
of Proposition~\ref{CaseIIaGen}).

\begin{equation}\label{Apdx:CaseIIaStrEqn_om}\small
	\begin{aligned}
		\ed\omega^0& = - [(3A_1+1)\omega^1+3(A_2 - A_4)\omega^2+3A_3\omega^3]\w\omega^0\\
			&\qquad\qquad 
			+\omega^1\w\omega^2+\omega^3\w\omega^4,\\
		\ed\omega^1 & = -3A_4^2\omega^0\w\omega^1+A_2\omega^1\w\omega^2
				+A_3\omega^1\w\omega^3+\omega^2\w\omega^3,\\
		\ed\omega^2& = -A_4\omega^0\w\omega^1 - 6A_4^2\omega^0\w\omega^2
				+\omega^0\w\omega^3 + (A_5 - 2A_1)\omega^1\w\omega^2
				\\
				&\qquad\qquad
				+A_6\omega^1\w\omega^3 + 2A_3\omega^2\w\omega^3,\\
		\ed\omega^3& = 3A_4^2\omega^0\w\omega^3 + A_1\omega^1\w\omega^3
						+A_2\omega^2\w\omega^3,\\
		\ed\omega^4& =-\phi\w\omega^3 -\omega^0\w\omega^1 -12A_4^2\omega^0\w\omega^4
			-(4A_1+1)\omega^1\w\omega^4 \\
			&\qquad\qquad
			+ (-4A_2+3A_4)\omega^2\w\omega^4
				-4A_3\omega^3\w\omega^4,
	\end{aligned}
\end{equation}
and
\begin{equation}\label{Apdx:CaseIIaStrEqn_As}\small
	\begin{aligned}
		\ed A_{1}& =  6 A_{4} \left[\left(A_{1}-A_{5}-\frac{1}{2}\right) A_{4}+\frac{A_{2}}{2}\right] \omega^0+A_{1,1} \omega^1\\
			&\qquad\qquad
			+\left[A_{2} \left(A_{5}-A_{1}\right)+3 A_{4}^{2}+A_{2,1}\right] \omega^2\\
			&\qquad\qquad
			+\left[-\frac{2}{3}+\frac{\left(6 A_{1}+1\right) A_{3}}{3}+\left(A_{2}-A_{4}\right) A_{6}+A_{3,1}\right] \omega^{3},
				\\%
		\ed A_{2}&= 9 \left[\left(-1-2 A_{5}\right) A_{4}^{2}+A_{2} A_{4}+\frac{2 A_{5,0}}{3}\right]A_{4} \omega^{0}+A_{2,1} \omega^{1}+A_{2,2} \omega^{2}\\
			&\qquad\qquad
			+\Big[-9 A_{6} \left(A_{5}+1\right) A_{4}^{2}+\left(6 A_{6} A_{2}+3 A_{6,2}-3\right)A_{4}-3 A_{5}^{2} \\
			&\qquad\qquad\qquad\quad
				+\left(3 A_{1}-1\right) A_{5}+3 A_{2} A_{3}+3 A_{6} A_{5,0}+A_{1}-3 A_{5,1}\Big] 
				\omega^{3},
				\\%
		\ed A_{3}&= \left[\left(-6 A_{5}-1\right) A_{4}-A_{2}\right] \omega^{0}+A_{3,1} \omega^{1}\\
		&\qquad\qquad
		+\Big[-9 A_{6} \left(A_{5}+1\right) A_{4}^{2}+\left(6 A_{6} A_{2}+3 A_{6,2}-3\right) A_{4}+3 A_{5} A_{1}\\
		&\qquad\qquad\qquad-3 A_{5}^{2}+3 A_{6} A_{5,0}-3 A_{5,1}\Big] \omega^{2}
		+A_{3,3} \omega^{3}+3 A_{4}^{2} \omega^{4},
				\\%
		\ed A_{4}& =  6 A_{4}^{3} \omega^{0}+\left[\left(-A_{5}+2 A_{1}\right) A_{4}+\frac{A_{2}}{3}\right] \omega^{1}\\
		&\qquad\qquad+\left[\left(-3 A_{5}-3\right) A_{4}^{2}+2 A_{2} A_{4}+A_{5,0}\right]\omega^{2}\\
		&\qquad\qquad
		+\left(2 A_{4} A_{3}-A_{5}-\frac{1}{3}\right) \omega^{3},
				\\%
		\ed A_{5}&= A_{5,0} \omega^{0}+A_{5,1} \omega^{1}+A_{5,2} \omega^{2}\\
			&\qquad\qquad
			+\left[-\frac{7}{3}+\frac{\left(3 A_{5}+2\right) A_{3}}{3}+2 \left(A_{2}-A_{4}\right) A_{6}+A_{6,2}\right] \omega^{3},
				\\%
		\ed A_{6}&= -\left(6 A_{6} A_{4}^{2}+2 A_{5}+\frac{4}{3}\right) \omega^{0}+A_{6,1} \omega^{1}+A_{6,2} \omega^{2}+A_{6,3} \omega^{3}+A_{4} \omega^{4}.
	\end{aligned}
\end{equation}

\subsection{Case IIb} With $\lambda \in \{\pm 1\}$, the structure equations read:
\begin{equation}\label{Apdx:CaseIIbStrEqn_om}\small
	\begin{aligned}
		\ed\omega^0& = 3 A_1 \omega^{0}\w\omega^{1}+\frac{7}{2\lambda} \omega^{0}\w\omega^{2}+3 A_2 \omega^{0}\w\omega^{3}+\omega^{1}\w\omega^{2}+\omega^{3}\w\omega^{4},\\
		\ed\omega^1& = \frac{7}{6\lambda}\omega^1\w\omega^2 + A_2\omega^1\w\omega^3+\omega^2\w\omega^3,\\
		\ed\omega^2& = \lambda\omega^1\w\omega^3- 2A_1\omega^1\w\omega^2+ 2A_2\omega^2\w\omega^3+\omega^0\w\omega^3,\\
		\ed\omega^3&= A_1\omega^1\w\omega^3 +\frac{7}{6\lambda}\omega^2\w\omega^3,\\
		\ed\omega^4&=-\phi\w\omega^3 - 4A_1\omega^1\w\omega^4 - \frac{14}{3\lambda}\omega^2\w\omega^4 - 4A_2\omega^3\w\omega^4 - \omega^0\w\omega^1,
	\end{aligned}
\end{equation}
and
\begin{equation}\label{Apdx:CaseIIbStrEqn_As}\small
	\begin{aligned}
		\ed A_1& =  A_{1,1}\omega^1 - \frac{7A_1}{6\lambda}\omega^2 + \left(2A_1A_2 + A_{2,1} + \frac{1}{2}\right)\omega^3,\\
		\ed A_2& =  -\frac{7}{6\lambda}\omega^0 + A_{2,1}\omega^1 - \frac{2A_1\lambda + 7A_2}{2\lambda}\omega^2+A_{2,3}\omega^3.
	\end{aligned}
\end{equation}

Similar to Case IIa, while these are not structure equations of an $e$-structure, 
reduction to an $e$-structure is possible by 
normalizing a derivative of $A_{2,3}$ (see the proof of 
Proposition~\ref{CaseIIbGen}.)

\subsection{Case III} The following are the $e$-structure equations for hyperbolic Monge--Amp\`ere equations satisfying $S_1 = 0$, $S_2\ne 0$ and $Q_1, Q_2\ne 0$. These equations
are involutive in the following sense: $\ed^2=0$ applied to $\omega^k$ $(k = 0,\ldots, 4)$
are identities, and the Pfaffian system generated by the nine $\ed A_\sigma - A_{\sigma,j}\omega^j$ $(\sigma = 1,2,\ldots,9)$ has absorbable torsion and involutive tableau.

\begin{equation}\label{Apdx:CaseIII_StrEqn_om}\small
	\begin{aligned}
		\ed\omega^0& = 3(A_2\omega^2 - A_4\omega^4)\w\omega^0 + \omega^1\w\omega^2+\omega^3\w\omega^4,\\
		\ed\omega^1&= - (A_2\omega^2+A_3\omega^3+A_4\omega^4)\w\omega^1 + \omega^2\w\omega^3,\\
		\ed\omega^2& =\left(\frac{A_4}{A_2}\omega^1 - \omega^3\right)\w\omega^0 
						+ (A_1+A_5)\omega^1\w\omega^2 + \frac{A_9-1}{3A_2}\omega^1\w\omega^3 \\
					&\qquad\qquad + A_6\omega^1\w\omega^4 - A_3\omega^2\w\omega^3+2A_4\omega^2\w\omega^4, \\
		\ed\omega^3& =\phantom{+} (A_1\omega^1+A_2\omega^2+A_4\omega^4)\w\omega^3 - \omega^1\w\omega^4,\\
		\ed\omega^4& = \left(\omega^1-\frac{A_2}{A_4}\omega^3\right)\w\omega^0
					+(-A_3+A_7)\omega^3\w\omega^4 + \frac{A_9+1}{3A_4}\omega^1\w\omega^3\\
					&\qquad\qquad - A_8 \omega^2\w\omega^3 - A_1\omega^1\w\omega^4
					+2A_2\omega^2\w\omega^4.
	\end{aligned}
\end{equation}
The $\ed A_{\sigma}$'s below are listed in the particular order as $\sigma = 1,3;2,4;5,7;6,8;9$.
\begin{equation}\label{Apdx:CaseIII_StrEqn_A1}\small
	\begin{aligned}
		\ed A_1& = \left(\frac{A_2}{A_4}+2A_4\right)\omega^0 + A_{1,1}\omega^1 + (A_1A_2 - A_8)\omega^2 \\
				&\qquad \qquad + A_{1,3}\omega^3 + (A_1A_4+2A_2A_6+2A_3 - A_7)\omega^4,
	\end{aligned}
\end{equation}
\begin{equation}\label{Apdx:CaseIII_StrEqn_A3}\small
	\begin{aligned}
		\ed A_3& = \left( \frac{A_4}{A_2}-2A_2\right)\omega^0 + A_{3,1}\omega^1 - (A_3A_4-A_6)\omega^4\\
				&\qquad\qquad + A_{3,3}\omega^3 + (-A_2A_3+2A_4A_8+2A_1+A_5)\omega^2;
	\end{aligned}
\end{equation}
\begin{equation}\label{Apdx:CaseIII_StrEqn_A2}\small
	\begin{aligned}
		\ed A_2& = -A_2(A_1+A_5)\omega^1  -A_2\left[4A_2 + \frac{1}{2A_4}+ A_{5,0}\frac{A_2}{A_4} + \left(\frac{A_2}{A_4}\right)^2 \right]\omega^2\\
				&\qquad\qquad +(A_4A_8 - A_2A_3)\omega^3+\left(2A_2A_4+\frac{1}{2}\right)\omega^4,
	\end{aligned}
\end{equation}
\begin{equation}\label{Apdx:CaseIII_StrEqn_A4}\small
	\begin{aligned}
		\ed A_4& = \phantom{+} A_4(A_3 - A_7)\omega^3 + A_4\left[ 4A_4+\frac{1}{2A_2} + A_{7,0}\frac{A_4}{A_2} - \left(\frac{A_4}{A_2}\right)^2\right]\omega^4\\
				&\qquad\qquad + (A_1A_4+A_2A_6)\omega^1 - \left(2A_2A_4+\frac{1}{2}\right)\omega^2;
	\end{aligned}
\end{equation}
\begin{equation}\label{Apdx:CaseIII_StrEqn_A5}\small
	\begin{aligned}
		\ed A_5& = A_{5,0}\omega^0+A_{5,1}\omega^1+A_{5,2}\omega^2+A_{5,3}\omega^3\\
				&
			+\left[A_6\left(\frac{A_2}{A_4}\right)^2 + A_{5,0}A_6\frac{A_2}{A_4}+
						A_8\frac{A_4}{A_2}+ \frac{A_6}{2A_4} - \frac{A_5}{2A_2}
						+A_4A_5 - 3A_3+A_7\right]\omega^4,
	\end{aligned}
\end{equation}
\begin{equation}\label{Apdx:CaseIII_StrEqn_A7}\small
	\begin{aligned}
		\ed A_7&= A_{7,0}\omega^0 + A_{7,1}\omega^1+A_{7,3}\omega^3+A_{7,4}\omega^4\\
				&
			+\left[ A_8\left(\frac{A_4}{A_2}\right)^2 - A_{7,0}A_8\frac{A_4}{A_2}+ A_6\frac{A_2}{A_4} + \frac{A_7}{2A_4} - \frac{A_8}{2A_2} - A_2A_7 + 3A_1 + A_5
			\right]\omega^2;
	\end{aligned}
\end{equation}
\begin{equation}\label{Apdx:CaseIII_StrEqn_A6}\small
	\begin{aligned}
		\ed A_6&= \left[-\left(\frac{A_4}{A_2}\right)^3+A_{7,0}\left(\frac{A_4}{A_2}\right)^2 - 1\right]\omega^0 + A_{6,1}\omega^1 + A_{6,4}\omega^4\\
				&
				+ \left[A_6\left(\frac{A_2}{A_4}\right)^2 + A_{5,0}A_6\frac{A_2}{A_4}+
						A_8\frac{A_4}{A_2}+ \frac{A_6}{2A_4} - \frac{A_5}{2A_2}+3A_2A_6\right]\omega^2\\
				&
				 +\frac{1}{3A_2}\left[ 3A_4(A_{3,1} - A_{1,3})
				 		+(A_9+1)A_{7,0}\frac{A_4}{A_2}
						+ 3A_1A_4(2A_3 - A_7) \right.
						\\
				&\qquad\qquad
						+ 3A_2A_6(3A_3 - A_7)
						 -3 A_4A_6A_8 + A_4(2A_9 - 3A_{7,1} +6 )
						 \\
				&\qquad\qquad
						\left.
						 - (A_9+1)\left(\frac{A_4}{A_2}\right)^2 + \frac{1}{A_2}\right]\omega^3,
	\end{aligned}
\end{equation}
\begin{equation}\label{Apdx:CaseIII_StrEqn_A8}\small
	\begin{aligned}
		\ed A_8&= \left[\left(\frac{A_2}{A_4}\right)^3 + A_{5,0}\left(\frac{A_2}{A_4}\right)^2+1\right]\omega^0 + A_{8,2}\omega^2 + A_{8,3}\omega^3\\
			& + \left[A_8\left(\frac{A_4}{A_2}\right)^2 - A_{7,0}A_8\frac{A_4}{A_2}+ A_6\frac{A_2}{A_4} + \frac{A_7}{2A_4} - \frac{A_8}{2A_2} - 3A_4A_8 \right]\omega^4\\
			& + \frac{1}{3A_4}\left[3A_2(A_{3,1} - A_{1,3})
					+(A_9-1)A_{5,0}\frac{A_2}{A_4}
					+3A_2A_3(2A_1+A_5) 
				\right.\\
			&\qquad\qquad
				-3A_4A_8(3A_1+A_5) - 3A_2A_6A_8
				+ A_2(2A_9 - 3A_{5,3} - 6)
					\\
			&\qquad\qquad
				\left.+ (A_9 - 1)\left(\frac{A_2}{A_4}\right)^2 - \frac{1}{A_4}\right]\omega^1;		
	\end{aligned}
\end{equation}
\begin{equation}\label{Apdx:CaseIII_StrEqn_A9}\small
	\begin{aligned}
		\ed A_9&= -3\left(\frac{A_2^2A_6}{A_4}+ \frac{A_4^2A_8}{A_2} + A_2A_5+A_4A_7\right)\omega^0 + A_{9,1}\omega^1+A_{9,3}\omega^3\\
				&+A_2\left[ \frac{1 - A_9}{A_2}\left(\left(\frac{A_2}{A_4}\right)^2 + A_{5,0}\frac{A_2}{A_4}+ \frac{1}{2A_4} \right)- 3A_3(2A_1+A_5) 
				\right.\\
				&\qquad\left.\phantom{\left(\frac{A_2}{A_2}\right)^2}+
				3A_6A_8 - 2A_9 + 3(A_{1,3} - A_{3,1} + A_{5,3}) + 9\right]\omega^2\\
				&+ A_4\left[\frac{1+A_9}{A_4}\left(-\left(\frac{A_4}{A_2}\right)^2+A_{7,0}\frac{A_4}{A_2} + \frac{1}{2A_2}\right)
					+ 3A_1(2A_3 - A_7)
				\right.\\
				&\qquad\left.\phantom{\left(\frac{A_4}{A_4}\right)^2} - 
				3A_6A_8 + 2A_9 - 3(A_{1,3} - A_{3,1}+A_{7,1}) + 9\right]\omega^4.
	\end{aligned}
\end{equation}

\end{appendix}

\bibliographystyle{alpha}
\normalbaselines 
\newcommand{\etalchar}[1]{$^{#1}$}

\end{document}